\documentclass[12pt]{amsart}
\usepackage{lineno}
\usepackage{amsfonts}
\usepackage{tipa}
\usepackage{amssymb}
\usepackage{mathrsfs}
\usepackage{amsthm,amsmath,amssymb}
\numberwithin{equation}{section}
\usepackage{txfonts}
\usepackage{graphicx}
\usepackage{geometry}
\usepackage{graphics}
\usepackage{caption}

\usepackage[colorlinks=true]{hyperref}
\usepackage{color}
\hypersetup{urlcolor=green, citecolor=blue}

\theoremstyle{plain}
\newtheorem{thm}{\hskip\parindent {Theorem}}[section]
\newtheorem{lem}{\hskip\parindent {Lemma}}[section]

\allowdisplaybreaks[4]

\begin{document}

\title{Asymptotic symmetry of solutions for reaction-diffusion equations via elliptic geometry}

\keywords{parabolic equation, asymptotic symmetry, monotonicity, elliptic geometry}
\subjclass[2010]{35R11,35B07}	
\author{Baiyu Liu and Wenlong Yang}	
\maketitle
\begin{center}
School of Mathematics and Physics\\
University of Science and Technology Beijing \\
30 Xueyuan Road, Haidian District
Beijing, 100083\\
P.R. China
\end{center}

\begin{abstract}
In this paper, we investigate the asymptotic symmetry and monotonicity of positive solutions to a reaction-diffusion equation in the unit ball, utilizing techniques from elliptic geometry. Firstly, we discuss the properties of solutions in the elliptic space. Then, we establish crucial principles, including the asymptotic narrow region principle.Finally, we employ the method of moving planes to demonstrate the asymptotic symmetry of the solutions.
\end{abstract}

\maketitle

\section{Introduction}
In this paper, we investigate the asymptotic symmetry and monotonicity of positive solutions to the following reaction-diffusion equation
\begin{equation}\label{main_problem}
\left\{\begin{array}{ll}
\frac{\partial u}{\partial t}(x,t)-\Delta u(x,t)=f(|x|,u(x,t),t), &(x,t)\in B_1(0)\times(0,\infty),  \\
u(x,t)=0 & (x,t)\in \partial B_1(0)\times(0,\infty),
\end{array}\right .
\end{equation}
where $B_1(0)$ the unit ball in $\mathbb{R}^N$, $N\geqslant2$. 

Reaction-diffusion equations are fundamental mathematical models that describe how the concentration of substances evolves over time under the influence of both chemical reactions and diffusion. They are widely applied in fields such as biology for pattern formation, chemistry for reaction waves, ecology for population dynamics, neuroscience for modeling brain activity, and physics for phase transitions and material science \cite{zhang_yi_wu-2024, jia_wang-2023}. These equations provide critical insights into self-organization in complex systems and form a theoretical basis for understanding various spatial structures and dynamic behaviors in nature. In terms of the application to the asymptotic symmetry of solutions, they also have been used in the proofs of convergence results for some autonomous and time-periodic equations\cite{chen_polacik-1996}.

For elliptic equations, the method of moving planes, initially introduced by Alexandrov \cite{alexandrov-1962} and Serrin \cite{serrin-1971}, and later developed by Berestycki and Nirenberg \cite{Bre1991}, Gidas, Ni, and Nirenberg \cite{gidas_ni_nirenberg-1979}, Chen and Li \cite{ChenLi1997}, among others, is a powerful tool for investigating the symmetry and monotonicity of solutions. 
In \cite{gidas_ni_nirenberg-1979}, Gidas, Ni and Nirenberg proved that if for each $u\in (0,\infty)$, the function $r\mapsto f(r,u):(0,1)\to \mathbb{R} $ is non-increasing, then any positive $C^2(\bar B_R(0))$ solution of 
\begin{equation}\label{elliptic_problem}
\left\{\begin{array}{ll}
\Delta u+f(|x|,u)=0, & x\in B_R(0),  \\
u=0, & x\in\partial B_R(0), 
\end{array}\right.
\end{equation}
is radially symmetric and decreasing in $r$. 
In recent years, several systematic approaches have emerged for studying symmetry and monotonicity in both local and nonlocal elliptic equations. These include the method of moving planes in integral form \cite{Chen2016, Dou2019, Guo2019, Guo2008}, the direct method of moving planes \cite{Bre1988,busca_sirakov-2000, Chen2017, LiuXu, Liu2018, LiNi1993}, the method of moving spheres \cite{ChLZh2017, Dou2017, Jin2011, LiZh1995, WeiXu1999}, and sliding methods \cite{Bre1991, LiuZ2021, Wu2020}. For further details on these methods, we refer to \cite{Chenbook} \cite{Chenbook2} \cite{CaoDai} \cite{Chen2015} and the references therein.

Notable, using hyperbolic geometry combine with the sprite of the moving plane method, a completely new monotonicity result can be achieved. Under the condition that $(1-r^2)^{(N+2)/2}f(r,(1-r^2)^{-(N-2)/2}u)$ is decreasing in $r\in(0,1)$ and fixed $u\in(0,\infty)$, Naito, Nishimoto and Suzuki have sequentially achieved each positive solution of \eqref{elliptic_problem} is radially symmetric and $(1-r^2)^{-(N-2)/2}u$ is decreasing in $r\in(0,1)$ in the cases of $N=2$ \cite{naito_nishimoto_suzuki-1996} and $N\geqslant2$ \cite{naito_suzuki-1998} . In \cite{shioji_watanabe-2012}, using not only hyperbolic geometry but also elliptic geometry, Shioji and Watanabe established the symmetry and monotonicity properties of a wide class of strong solutions of \eqref{elliptic_problem}.

As to the parabolic equations, the situation becomes significantly more intricate. There have been some preliminary study of symmetric solutions of periodic-parabolic problem \cite{babin_sell-2000,dancer_hess-1994,foldes-2011} and further results concerning entire solutions where time variable $t\in\mathbb{R}$\cite{babin-1994,polacik-2006}. In a different type of symmetry considering the Cauchy-Dirichlet problem of reaction-diffusion equations, asymptotic symmetry shows a tendency of positive solutions to improve their symmetry as time variable $t\in(0,\infty)$ increases, becoming "symmetric and monotone in the limit" as $t\to\infty$. In this context, Li\cite{li-1989} obtained symmetry for positive solutions with symmetric initial solutions. Without that, in bounded, symmetric and strictly convex domain $\Omega$, Hess and Poláčik \cite{hess_polacik-1994} showed the asymptotic symmetry and monotonicity of positive solutions to the following problem  
\begin{equation}
\left\{\begin{array}{ll}
\partial_t u=\Delta u(x,t)+f(u(x,t),t), & (x,t)\in\Omega\times(0,\infty),  \\
u(x,t)=0, & (x,t)\in\partial\Omega\times(0,\infty).
\end{array}\right.
\end{equation}
 It was assumed in \cite{hess_polacik-1994} that $f$ is uniformly Lipschitz-continuous in $u$ and Hölder-continuous of exponent $(\alpha,(\alpha/2))$ with respect to $(u,t)$. Subsequently, in both bounded and unbounded domains, Poláčik \cite{polacik-2007,polacik-2005} extended the result to the following generalized fully nonlinear parabolic equation
\begin{equation}\label{generaliazed_parabolic_problem}
\left\{\begin{array}{ll}
u_t=F(t,x,u,Du,D^{2}u), & (x,t)\in\Omega\times(0,\infty), \\
u(x,t)=0 & (x,t)\in\partial\Omega\times(0,\infty).
\end{array}\right.
\end{equation}
In an independent work, Babin and Sell \cite{babin_sell-1995} gave a similar result. With the symmetry conclusion of entire solutions of \eqref{generaliazed_parabolic_problem}, Poláčik \cite{polacik-2006} gave a survey that summarized the following limitations of existing asymptotic symmetry results: the regularity requirements of the time-dependence of the nonlinearities and domain, the compactness requirements of spatial derivatives and the strong positivity requirements on the nonlinearities in the case of nonsmooth domains\cite{polacik-2009}. Moreover, Saldaña and Weth \cite{saldana_weth-2012} established the asymptotic foliated Schwarz symmetry which indicates that all positive solutions of \eqref{main_problem} become axially symmetric with respect to a common axis passing through the origin as $t\to\infty$.

The motivation for this paper is to extend the results of \cite{naito_nishimoto_suzuki-1996, naito_suzuki-1998, shioji_watanabe-2012} to the framework of reaction-diffusion equations. Our approach, which integrates elliptic geometry with the method of moving planes, is inspired by the work of \cite{chen_wang_niu-2021} and \cite{shioji_watanabe-2012}.

In order to clarify the theorem, we first introduce the $\omega$-limit-set of $u$:
\begin{equation}\label{omegau_definition}
\omega(u):=\left\{\varphi\in C_0(\overline{B_1(0)})\,|\,\exists \,t_k\to\infty \,\,\,such \,\,that\,\, \varphi=\lim_{k\to\infty}u(\cdot,t_k)\right\}.
\end{equation}
By the discussion of $u$ in Appendix, the orbit $\{u(\cdot,t):t>0\}$ is relatively compact in $C_0(\overline{B_1(0)})$ and that $\omega(u)$ is a nonempty compact subset of $C_0(\overline{B_1(0)})$.

From now on, we will assume the nonlinearity $f:(0,1)\times(0,\infty)\times[0,\infty)\to\mathbb{R}$ satisfy
\begin{enumerate}
	\item[(F1)]
	for each $M>0$, $f(r,u,t)$ is Lipschitz-continuous in $u$ uniformly with respect to $t$ and $r$ in the region $(0,1)\times[-M,M]\times[0,\infty)$;
	\item[(F2)]
	for each $r\in (0,1)$ and $\tau>0$ there exist a constant $H$ and a small constant $\varepsilon_0$ both independent of $\tau$ such that
	\begin{equation*}
		|f(r_1,u,t_1)-f(r_2,u,t_2)|\leqslant H(|r_1-r_2|^\alpha+|t_1-t_2|^{\frac{\alpha}{2}}),
	\end{equation*}
	for all $u\in(0,\infty)$ and $(r_1,t_1),\,(r_2,t_2)\in (0,1)\times[\tau-\varepsilon_0,\tau+\varepsilon_0]$;
	\item[(F3)]
	for all $t\in[0,\infty)$ and each fixed $u\in(0,\infty)$, $(1+r^2)^{\frac{N+2}{2}}f(r,(1+r^2)^{-\frac{N-2}{2}}u,t)$ is decreasing in $r\in (0,1)$.
\end{enumerate}

\begin{thm}\label{main_theorem} 
Let $f$ satisfy conditions (F1), (F2) and (F3).
Assume $u\in C^{2,1}(B_1(0)\times(0,\infty))\cap C(\overline{B_1(0)}\times[0,\infty))$ is a positive bounded solution of \eqref{main_problem} that satisfies $\partial u/\partial t$ is non-negative and bounded, $\nabla u(x,t)$ is bounded for all $(x,t)\in B_1(0)\times(0,\infty)$. 
Then for each $\varphi(x)\in\omega(u)$, the following holds:
\begin{itemize}
	\item Either
	$\varphi(x)\equiv 0$;
	\item Or
	$\varphi(x)$ is radially symmetric about the origin and satisfies $\partial_{r}\left((1+|r|^2)^{\frac{N-2}{2}}\varphi(r)\right)<0$ for $r=|x|\in(0,1)$, where $x\in B_1(0)$.
\end{itemize}
\end{thm}

The structure of the paper is as follows: Section \ref{sec:pre} presents the preliminary results that form the foundation for our main findings. In Section \ref{sec:maxp} , we introduce the asymptotic narrow region principle, which plays a crucial role in the proof of our main theorem.  The proof of Theorem \ref{main_theorem} is provided in detail in Section \ref{sec:proof}. Finally, the appendix includes additional properties of the solutions.

\section{Preliminaries}\label{sec:pre}

In this section, we will establish some essential preliminaries for applying the moving plane method in the space $(B_1(0), g)$, where the metric tensor $g$ is defined by
\begin{equation}
\label{def:2.1}
\frac{4|dx|^2}{(1+|x|^2)^2}.
\end{equation}
Here, $|\cdot|$ denotes the standard Euclidean norm, consistent with the notation used in other sections. For each $\lambda \in (0,1)$, let $T_\lambda \subset B_1(0)$ be a totally geodesic plane intersecting the $x_1$-axis orthogonally at the point $(\lambda, 0, \dots, 0)$. It follows that
\begin{equation}\label{Tlambda_definition}
T_\lambda=\left\{x\in B_1(0):|x-e_\lambda|=\left(\frac{1+\lambda^2}{2\lambda}\right)\right\},
\end{equation}
where
\begin{equation}\label{elambda_definition}
e_\lambda=\left(\left(\frac{1-\lambda^2}{-2\lambda}\right)\,,0\,,\cdots,0\right).
\end{equation}
Define
\begin{equation}\label{sigma_definition}
\Sigma_\lambda=\left\{x\in B_1(0)\ \vline\  |x-e_\lambda|>\left(\frac{1+\lambda^2}{2\lambda}\right)\right
\}.
\end{equation}
For each $x \in \Sigma_\lambda$, let $x^\lambda$ denote the reflection of $x$ with respect to $T_\lambda$ in the space $(B_1(0), g)$. This reflection can be expressed as
\begin{equation}\label{xlambda_definition}
x^\lambda=e_\lambda+\left(\frac{1+\lambda^2}{2\lambda}\right)^{2}\frac{x-e_\lambda}{|x-e_\lambda|^2}.
\end{equation}
We remark that 
\begin{equation}\label{xlambda_distance}
|x|^2>|x^\lambda|^2.
\end{equation}
The proof is presented in Lemma \ref{proof_of_xlambda_distance}.

The Laplace-Beltrami operator $\Delta_{(g,x)}$ on the space $(B_1(0),g)$ at $x\in B_1(0)$ is defined by
\begin{equation}\label{operatordefinition}
\Delta_{(g,x)}=\left(\frac{1+|x|^2}{2}\right)^2\left(\Delta-\frac{2(N-2)}{1+|x|^2}\sum_{i=1}^{N}x_i\frac{\partial}{\partial x_i}\right),
\end{equation}
where $\Delta=\sum_{i=1}^{N}\frac{\partial^2}{\partial x_i^2}$.

Let $u(x,t)$ be a solution to the parabolic problem given by equation \eqref{main_problem}.  For each $\lambda\in(0,1)$, we introduce new functions  $v(x,t)$, $w_\lambda(x,t)$ and $z_\lambda(x,t)$ to compare the value of $u(x,t)$ with $u(x^\lambda,t)$ and to simplify the analysis of the gradient's impact. These functions are defined as follows:
\begin{eqnarray}
\label{v_definition}&v(x,t)=(1+|x|^2)^\frac{N-2}{2}u(x,t),\qquad (x,t)\in B_1(0)\times(0,\infty),\\ \label{wlambda_definition}&w_\lambda(x,t)=v(x^\lambda,t)-v(x,t),\qquad (x,t)\in\Sigma_\lambda\times(0,\infty),\\ \label{zlambda_definition}&z_\lambda(x,t)=(1+|x|^2)^{-\frac{N-2}{2}}w_\lambda(x,t),\qquad (x,t)\in\Sigma_\lambda\times(0,\infty).
\end{eqnarray}
From \eqref{v_definition}-\eqref{zlambda_definition}, for $(x,t)\in\Sigma_\lambda\times(0,\infty)$, we obtain 
\begin{equation}\label{zlambda_ulambda_u}
\begin{aligned}
z_\lambda(x,t) &= (1+|x|^2)^{-\frac{N-2}{2}}((1+|x^\lambda|^2)^\frac{N-2}{2}u(x^\lambda,t)-(1+|x|^2)^\frac{N-2}{2}u(x,t)) \\ &=\left(\frac{1+|x^\lambda|^2}{1+|x|^2}\right)^\frac{N-2}{2}u(x^\lambda,t)-u(x,t).
\end{aligned}
\end{equation}
Clearly, $z_\lambda\in(C^{2,1}(\Sigma_\lambda\times(0,\infty))\cap C(\overline{\Sigma_\lambda}\times[0,\infty)])$ and $z_\lambda=0$ on $T_\lambda$.
For each $\varphi(x)\in\omega(u)$, denote
\begin{equation}\label{psilambda_definition}
\begin{aligned}
\psi_\lambda(x) &= (1+|x|^2)^{-\frac{N-2}{2}}((1+|x^\lambda|^2)^\frac{N-2}{2}\varphi(x^\lambda)-(1+|x|^2)^\frac{N-2}{2}\varphi(x)) \\ &=\left(\frac{1+|x^\lambda|^2}{1+|x|^2}\right)^\frac{N-2}{2}\varphi(x^\lambda)-\varphi(x),
\end{aligned}    
\end{equation}
which is an $\omega$-limit of $z_\lambda(x,t)$.

By virtue of the definitions given in (\ref{operatordefinition}) and \eqref{v_definition}, we observe that for $x\in B_1(0)$, the function $v$ satisfies
\begin{equation}\label{v_equation}
\begin{aligned}
\Delta_{(g,x)}v(x,t) &= \left(\frac{1+|x|^2}{2}\right)^2\left(\Delta v(x,t)-\frac{2(N-2)}{1+|x|^2}\sum_{i=1}^{N}x_i\frac{\partial v}{\partial x_i}\right) \\ &= \frac{(1+|x|^2)^{\frac{N+2}{2}}}{4}\Delta u(x,t)+\frac{N(N-2)}{4}(1+|x|^2)^{\frac{N-2}{2}}u(x,t) 
\\ &= \frac{(1+|x|^2)^{\frac{N+2}{2}}}{4}\left(\frac{\partial u}{\partial t}-f(|x|,u(x,t),t)\right)+\frac{N(N-2)}{4}v(x,t) \\ &= \left(\frac{1+|x|^2}{2}\right)^{2}\frac{\partial v}{\partial t}-\frac{(1+|x|^2)^{\frac{N+2}{2}}}{4}f(|x|,(1+|x|^2)^{-\frac{N-2}{2}}v(x,t),t)+\frac{N(N-2)}{4}v(x,t).
\end{aligned}
\end{equation}

In addition to the above notations, we now present the following lemmas to establish the properties of 
$z_\lambda(x,t)$.
\begin{lem}\label{lemma_zlambda_inequality}
Let $u\in C^{2,1}(B_1(0)\times(0,\infty))\cap C(\overline{B_1(0)}\times[0,\infty))$ be a positive bounded solution of \eqref{main_problem} that satisfies $\frac{\partial u}{\partial t}(x,t)\geqslant 0$ for all $(x,t)\in B_1(0)\times(0,\infty)$.
Assume that the function $(1+r^2)^{\frac{N+2}{2}}f(r,(1+r^2)^{-\frac{N-2}{2}}s,t)$ is nonincreasing in $r\in(0,1)$ for each fixed $s\in(0,\infty)$ and $t\in(0,\infty)$. Under these conditions, $z_\lambda$ satisfies
\begin{equation}\label{zlambda_inequality}
\frac{\partial z_\lambda}{\partial t}-\Delta z_\lambda \geqslant c_\lambda(x,t)z_\lambda
\end{equation}
in $\Sigma_\lambda\times(0,\infty)$, where
\begin{equation}
c_\lambda(x,t)=\frac{f(|x|,(1+|x|^2)^{-\frac{N-2}{2}}v(x^\lambda,t),t)-f(|x|,(1+|x|^2)^{-\frac{N-2}{2}}v(x,t),t)}{(1+|x|^2)^{-\frac{N-2}{2}}v(x^\lambda,t)-(1+|x|^2)^{-\frac{N-2}{2}}v(x,t)}.
\end{equation}
\end{lem}
\begin{proof}
Let $\lambda\in(0,1)$ , $x\in\Sigma_\lambda$ and set $y=x^\lambda$. Since the Laplace-Beltrami operator is invariant under the isometry, as shown in Lemma \ref{proof_of_operater_identity} in the Appendix, we have 
\begin{equation}\label{operator_identity}
\Delta_{(g,y)}v(y,t)=\Delta_{(g,x)}v(x^\lambda,t).
\end{equation}
Using this equality, together with \eqref{xlambda_distance} and the monotonicity assumption (F3) of $(1+r^2)^{\frac{N+2}{2}}f(r,(1+r^2)^{-\frac{N-2}{2}}s,t)$, we deduce that
{\setlength\arraycolsep{2pt}
\begin{eqnarray*}
0 &= & \Delta_{(g,y)}v(y,t)-\frac{N(N-2)}{4}v(y,t)+\frac{(1+|y|^2)^{\frac{N+2}{2}}}{4}f(|y|,(1+|y|^2)^{-\frac{N-2}{2}}v(y,t),t)-\left(\frac{1+|y|^2}{2}\right)^{2}\frac{\partial v(y,t)}{\partial t} \\ && -\Delta_{(g,x)}v(x,t)+\frac{N(N-2)}{4}v(x,t)-\frac{(1+|x|^2)^{\frac{N+2}{2}}}{4}f(|x|,(1+|x|^2)^{-\frac{N-2}{2}}v(x,t),t)+\left(\frac{1+|x|^2}{2}\right)^{2}\frac{\partial v(x,t)}{\partial t} \\ &=& \Delta_{(g,x)}w_\lambda(x,t)-\frac{N(N-2)}{4}w_\lambda(x,t) \\ && +\frac{(1+|x^\lambda|^2)^{\frac{N+2}{2}}}{4}f(|x^\lambda|,(1+|x^\lambda|^2)^{-\frac{N-2}{2}}v(x^\lambda,t),t)-\frac{(1+|x|^2)^{\frac{N+2}{2}}}{4}f(|x|,(1+|x|^2)^{-\frac{N-2}{2}}v(x,t),t) \\ && -\left(\frac{1+|x|^2}{2}\right)^{2}\frac{\partial w_\lambda(x,t)}{\partial t} +\frac{\partial v(y,t)}{\partial t}\left(\left(\frac{1+|x|^2}{2}\right)-\left(\frac{1+|x^\lambda|^2}{2}\right)\right) \\ &\geqslant& \Delta_{(g,x)}w_\lambda(x,t)-\frac{N(N-2)}{4}w_\lambda(x,t)-\left(\frac{1+|x|^2}{2}\right)^{2}\frac{\partial w_\lambda(x,t)}{\partial t} \\ && +\frac{(1+|x|^2)^{\frac{N+2}{2}}}{4}f(|x|,(1+|x|^2)^{-\frac{N-2}{2}}v(x^\lambda,t),t)-\frac{(1+|x|^2)^{\frac{N+2}{2}}}{4}f(|x|,(1+|x|^2)^{-\frac{N-2}{2}}v(x,t),t) \\ &=& \Delta_{(g,x)}w_\lambda(x,t)-\frac{N(N-2)}{4}w_\lambda(x,t)+\left(\frac{1+|x|^2}{2}\right)^{2}c_\lambda(x,t)w_\lambda(x,t)-\left(\frac{1+|x|^2}{2}\right)^{2}\frac{\partial w_\lambda(x,t)}{\partial t} \\ &=& \left(\frac{1+|x|^2}{2}\right)^{2}\left(-\frac{\partial w_\lambda(x,t)}{\partial t}+\Delta w_\lambda -\frac{2(N-2)}{1+|x|^2}\sum_{i=1}^{N}x_i\frac{\partial w_\lambda}{\partial x_i}-\frac{N(N-2)}{(1+|x|^2)^2}w_\lambda(x,t)+c_\lambda(x,t)w_\lambda(x,t)\right).
\end{eqnarray*}}
Thus we obtain
\begin{equation}
\frac{\partial w_\lambda(x,t)}{\partial t}-\Delta w_\lambda+\frac{2(N-2)}{1+|x|^2}\sum_{i=1}^{N}x_i\frac{\partial w_\lambda}{\partial x_i}+\frac{N(N-2)}{(1+|x|^2)^2}w_\lambda(x,t)\geqslant c_\lambda(x,t)w_\lambda(x,t).
\end{equation}
From \eqref{zlambda_definition}, an elementary computation shows that
\begin{eqnarray*}
\frac{\partial w_\lambda(x,t)}{\partial t} &=& (1+|x|^2)^{\frac{N-2}{2}}\frac{\partial z_\lambda}{\partial t}, \\ \Delta z_\lambda(x,t) &=& (\Delta(1+|x|^2)^{-\frac{N-2}{2}})w_\lambda+2\nabla(1+|x|^2)^{-\frac{N-2}{2}})\cdot\nabla w_\lambda+(1+|x|^2)^{-\frac{N-2}{2}})\Delta w_\lambda \\ &=& (1+|x|^2)^{-\frac{N-2}{2}}\left(\Delta w_\lambda-\frac{2(N-2)}{1+|x|^2}\sum_{i=1}^{N}x_i\frac{\partial w_\lambda}{\partial x_i}-\frac{N(N-2)}{(1+|x|^2)^2}w_\lambda(x,t)\right).
\end{eqnarray*}
Therefore, $z_\lambda$, as defined ub \eqref{zlambda_definition}, satisfies the inequality given in \eqref{zlambda_inequality}.
\end{proof}

\begin{lem}\label{lemma_unit_outer_normal}
Assume that for some $(x_0,t_0)\in T_\lambda\times(0,\infty)$, there holds
\begin{equation*}
\frac{\partial z_\lambda}{\partial n}(x_0,t_0)<0,
\end{equation*}
where $n$ denotes the unit outer normal to $\partial\Sigma_\lambda$. Then,
\begin{equation}\label{v_unit_outer_normal}
\frac{\partial v}{\partial n}(x_0,t_0)>0.
\end{equation}
\end{lem}
\begin{proof}
Since $z_\lambda(x,t)=0$ on $T_\lambda\times(0,\infty)$, we have 
\begin{equation}\label{wlambda_zlambda_unit_outer_normal_comparison}
\frac{\partial w_\lambda}{\partial n}(x_0,t_0)=(1+|x_0|^2)^{\frac{N-2}{2}}\frac{\partial z_\lambda}{\partial n}(x_0,t_0)<0.
\end{equation}
We define $x_p$ as $x_p=x_0-pn(x_0)$, $p>0$, and $n(x_0)$ is the  unit outer normal to $\partial\Sigma_\lambda$ at point $x_0$. Specifically, $n(x_0)=-(x_0-e_\lambda)/|x_0-e_\lambda|$. 

For $x_p\in \Sigma_\lambda$, using the definition in \eqref{xlambda_definition}, we find that 
\begin{equation*}
x_p^\lambda=x_q=x_0+qn(x_0),
\end{equation*}
where 
\begin{equation*}
q=\frac{(1+\lambda^2)p}{1+\lambda^2+2p\lambda}>0.
\end{equation*}
This result follows from the property given in \eqref{xlambda_definition} that
\begin{equation*}
|x_p-e_\lambda|=\frac{1+\lambda^2}{2\lambda}+p,\quad|x_q-e_\lambda|=\frac{1+\lambda^2}{2\lambda}-q,\quad|x_p-e_\lambda||x_q-e_\lambda|=\left(\frac{1+\lambda^2}{2\lambda}\right)^2.
\end{equation*}
Then, it follows that
\begin{eqnarray*}
\frac{\partial w_\lambda(x_0,t)}{\partial n} &=&\lim_{p \to 0^+}\frac{w_\lambda(x_p,t)-w_\lambda(x_0,t)}{-p}\\&=& \lim_{p \to 0^+}\left(\frac{v(x_q,t)-v(x_0,t)}{-p}+\frac{v(x_p,t)-v(x_0,t)}{p}\right) \\ &=& \lim_{q \to 0^+}\frac{v(x_q,t)-v(x_0,t)}{-q}\left(\frac{q}{p}\right)+\lim_{p \to 0^+}\frac{v(x_p,t)-v(x_0,t)}{p} \\ &=& -2\frac{\partial v(x_0,t)}{\partial n}.
\end{eqnarray*}
From \eqref{wlambda_zlambda_unit_outer_normal_comparison}, we conclude that inequality \eqref{v_unit_outer_normal} holds.
\end{proof}

\section{Asymptotic Maximum Principles}\label{sec:maxp}
In this section, we present the following asymptotic narrow region principle, which will play a crucial role in establishing Theorem \ref{main_theorem}.
\begin{lem}\label{asymptotic_narrow_region_principle}(Asymptotic narrow region principle)
Assume that $\Omega$ is a bounded narrow region with respect to $e_\lambda$ contained within the annulus defined by
\begin{equation}\label{narrow_annulus_definition}
\left\{x\in B_1(0)\ \vline\ \frac{1+\lambda^2}{2\lambda}<|x-e_\lambda|<\frac{1+\lambda^2}{2\lambda}+\delta\right\}
\end{equation}
for some small $\delta>0$, where $e_\lambda$ is defined by \eqref{elambda_definition}.

For sufficiently large $\overline{t}$, assume that $z_\lambda(x,t)\in C^2(\Omega)\times C^1([\,\overline{t},\infty])$ is bounded and lower semi-continuous in $x$ on $\overline{\Omega}$, and satisfies
\begin{equation}\label{problem_asymptotic_narrow_region_principle}
\left\{\begin{aligned}
&\frac{\partial z_\lambda(x,t)}{\partial t}-\Delta z_\lambda(x,t)\geqslant c_\lambda(x,t)z_\lambda(x,t), \qquad &(x,t)\in\Omega\times[\,\overline{t},\infty),  \\ &z_\lambda(x,t)\geqslant0, \qquad
& (x,t)\in\partial\Omega\times[\,\overline{t},\infty),
\end{aligned}\right.
\end{equation}
where $c_\lambda(x,t)$ is bounded from above. Then for sufficiently small $\delta$ the following statement holds:
\begin{equation}\label{result_asymptotic_narrow_region_principle}
\varliminf_{t\to \infty}z_\lambda(x,t)\geqslant0,\quad\forall x\in\Omega.
\end{equation}
\end{lem}
\begin{proof}
Let $m$ be a fixed positive constant  that will be determined later.
Define
\begin{equation*}
\Tilde{z}_\lambda(x,t)=\frac{e^{mt}z_\lambda(x,t)}{\phi(x)}.
\end{equation*}
Then  $\Tilde{z}_\lambda(x,t)$  satisfies
\begin{eqnarray*}
\frac{\partial\Tilde{z}_\lambda(x,t)}{\partial t} &=& m\frac{e^{mt}z_\lambda(x,t)}{\phi(x)}+\frac{e^{mt}}{\phi(x)}\frac{\partial z_\lambda(x,t)}{\partial t},
\end{eqnarray*}
and
\begin{eqnarray*}
 \Delta\Tilde{z}_\lambda(x,t)&=&e^{mt}\left(\frac{\Delta z_\lambda(x,t)}{\phi(x)}-2\frac{\nabla\Tilde{z}_\lambda(x,t)\nabla\phi(x)}{\phi(x)}-\frac{\Delta\phi(x)}{\phi(x)}\frac{z_\lambda(x,t)}{\phi(x)}\right).
\end{eqnarray*}
Thus, we find that
\begin{equation}\label{tildezlambda_inequality}
\frac{\partial\Tilde{z}_\lambda(x,t)}{\partial t}-\Delta\Tilde{z}_\lambda(x,t)-2e^{mt}\frac{\nabla\Tilde{z}_\lambda(x,t)\nabla\phi(x)}{\phi(x)}\geqslant\left(c_\lambda(x,t)+\frac{\Delta\phi(x)}{\phi(x)}+m\right)\Tilde{z}_\lambda(x,t).
\end{equation}
Let $\phi(x)$ be defined as
\begin{equation*}
\phi(x)=\sin\left(\frac{|x-e_\lambda|-\frac{1+\lambda^2}{2\lambda}}{\delta}+\frac{\pi}{2}\right)=\sin\left(\frac{\sqrt{\left|x_1-\frac{1-\lambda^2}{-2\lambda}\right|^2+|x_2|^2+\cdots+|x_N|^2}-\frac{1+\lambda^2}{2\lambda}}{\delta}+\frac{\pi}{2}\right).
\end{equation*}
For each $x\in\Omega$, we have 
\begin{equation*}
\frac{\pi}{2}<\frac{|x-e_\lambda|-\frac{1+\lambda^2}{2\lambda}}{\delta}+\frac{\pi}{2}<1+\frac{\pi}{2}<\pi.
\end{equation*}
Direct calculation shows
\begin{eqnarray*}
\Delta\phi(x)&=&-\frac{1}{\delta^2}\sin\left(\frac{|x-e_\lambda|-\frac{1+\lambda^2}{2\lambda}}{\delta}+\frac{\pi}{2}\right)+\cos\left(\frac{|x-e_\lambda|-\frac{1+\lambda^2}{2\lambda}}{\delta}+\frac{\pi}{2}\right)\frac{1}{|x-e_\lambda|\delta}(n-1)\\&<&-\frac{1}{\delta^2}\sin\left(\frac{|x-e_\lambda|-\frac{1+\lambda^2}{2\lambda}}{\delta}+\frac{\pi}{2}\right).
\end{eqnarray*}
Thus,
\begin{equation}\label{phi_equality}
\frac{\Delta\phi(x)}{\phi(x)}<-\frac{1}{\delta^2}.
\end{equation}

We claim that for any $T>\overline{t}$,
\begin{equation}\label{tildezlambda_geq_min0inf}
\Tilde{z}_\lambda(x,t)\geqslant\min\left\{0,\inf_{\Omega}\Tilde{z}_\lambda(x,\overline{t})\right\},\quad(x,t)\in\Omega\times[\overline{t},T].
\end{equation}
If \eqref{tildezlambda_geq_min0inf} is not true, by \eqref{problem_asymptotic_narrow_region_principle} and the lower semi-continuity of $z_\lambda$ on $\overline{\Omega}\times[\overline{t},T]$, there exists $(x_0,t_0)$ in $\Omega\times(\overline{t},T]$ such that
\begin{equation}
\Tilde{z}_\lambda(x_0,t_0)=\min_{\overline{\Omega}\times(\overline{t},T]}\Tilde{z}_\lambda(x,t)<\min\{0,\inf_{\Omega}\Tilde{z}_\lambda(x,\Tilde{t})\}.
\end{equation}
Since
\begin{equation*}
\Tilde{z}_\lambda(x,t)=0,\quad(x,t)\in \overline{\Omega}\cap T_\lambda,
\end{equation*}
and
\begin{equation*}
\Tilde{z}_\lambda(x,t)>0,\quad(x,t)\in \overline{\Omega}\cap \partial{B_1(0)},
\end{equation*}
where $T_\lambda$ is defined in \eqref{Tlambda_definition}, the minimum point $x_0$ is an interior point of $\Omega$. Therefore,
\begin{eqnarray}
\label{tildezlambda_partialt}\frac{\partial\Tilde{z}_\lambda(x_0,t_0)}{\partial t}&\leqslant&0, \\ \label{laplace_tildezlambda}\Delta\Tilde{z}_\lambda(x_0,t_0)&\geqslant&0, \\ \label{gradient_tildezlambda}\nabla\Tilde{z}_\lambda(x_0,t_0)&=&0.
\end{eqnarray}
From \eqref{tildezlambda_inequality}, \eqref{phi_equality} and \eqref{tildezlambda_partialt}-\eqref{gradient_tildezlambda}, we have 
\begin{equation}\label{tildezlambda_leq_0}
0\geqslant\frac{\partial\Tilde{z}_\lambda(x_0,t_0)}{\partial t}-\Delta\Tilde{z}_\lambda(x_0,t_0)>\left(c_\lambda(x_0,t_0)-\frac{1}{\delta^2}+m\right)\Tilde{z}_\lambda(x_0,t_0).
\end{equation}
Since $c_\lambda(x,t)$ is bounded from above for all $(x,t)\in B_1(0)\times(0,\infty)$, we can choose $\delta$ small enough such that
\begin{equation*}
c_\lambda(x_0,t_0)-\frac{1}{\delta^2}<-\frac{1}{2\delta^2}.
\end{equation*}
Taking $m=1/2\delta^2$, we derive that the right hand side of the second inequality of \eqref{tildezlambda_leq_0} is strictly greater than $0$, since $\Tilde{z}_\lambda(x_0,t_0)<0$. This contradicts with \eqref{tildezlambda_leq_0}. 

Therefore, given the boundedness of $z_\lambda$, there exists a constant $C>0$ such that
\begin{equation*}
\Tilde{z}_\lambda(x,t)\geqslant\min\{0,\inf_{\Omega}\Tilde{z}_\lambda(x,\overline{t})\}\geqslant-C,\quad(x,t)\in\Omega\times[\overline{t},T].
\end{equation*}
Thus, we have
\begin{equation*}
z_\lambda(x,t)\geqslant e^{-mt}(-C),\quad\forall t>\overline{t}.
\end{equation*}
Taking the limit $t\to\infty$, we obtain
\begin{equation*}
\varliminf_{t\to \infty}z_\lambda(x,t)\geqslant0,\quad\forall x\in\Omega.
\end{equation*}

This complete the proof of Lemma \ref{asymptotic_narrow_region_principle}.
\end{proof}

In the proof of the main theorem, we will also use the following two classical maximum principles. For the convenience of the reader, we now show them in the version suitable for this article.

\begin{lem}[Strong parabolic maximum principle for not necessarily non-negative coefficient]
\label{strong_parabolic_maximum_principle}
For each $\lambda\in(0,1)$, assume $z(x,t)\in C^{2,1}\left(\Sigma_\lambda\times(0,\infty)\right)\cap C\left(\overline{\Sigma_\lambda}\times[0,\infty)\right)$ and for $(x,t)\times\Sigma_\lambda\times(0,\infty)$, $z(x,t)\geqslant0$ satisfies that
\begin{equation*}
\frac{\partial z(x,t)}{\partial t}-\Delta z(x,t)-c(x,t)z(x,t)\geqslant 0,
\end{equation*}
where $c(x,t)$ is bounded in $\Sigma_\lambda\times(0,\infty)$. If $z(x,t)$ attains its minimum $0$ over $\overline{\Sigma_\lambda}\times [0,\infty)$ at a point $(x_0,t_0)\in\Sigma_\lambda\times(0,\infty)$, then $z(x,t)\equiv0$ in $\Sigma_\lambda\times(0,t_0]$.
\end{lem}
\begin{proof}
For each $\lambda\in(0,1)$, we set $c_0=\mathop{sup}\limits_{\Sigma_\lambda\times(0,\infty)}c(x,t)$ and let 
\begin{equation*}
\overline{z}(x,t)=e^{-c_0t}z(x,t),
\end{equation*}
then for $(x,t)\times\Sigma_\lambda\times(0,\infty)$, $\overline{z}(x,t)=e^{-c_0t}z(x,t)\geqslant 0$ and satisfies that 
\begin{equation*}
\begin{aligned}
&\quad\frac{\partial\overline{z}(x,t)}{\partial t}-\Delta\overline{z}(x,t)+(c_0-c(x,t))\overline{z}(x,t)\\&=-c_0e^{-c_0t}z(x,t)+e^{-c_0t}\frac{\partial z(x,t)}{\partial t}-e^{-c_0t}\Delta z(x,t)+c_0e^{-c_0t}z(x,t)-c(x,t)e^{-c_0t}z(x,t)\\&=e^{-c_0t}\left(\frac{\partial z(x,t)}{\partial t}-\Delta z(x,t)-c(x,t)z(x,t)\right)\geqslant 0.
\end{aligned}
\end{equation*}
If $z(x,t)$ attains its minimum $0$ at $(x_0,t_0)$, then $\overline{z}(x,t)$ also gets its minimum $0$ at the point $(x_0,t_0)$ over $\overline{\Sigma_\lambda}\times[0,\infty)$. Considering $c_0-c(x,t)\geqslant 0$ in $\Sigma_\lambda\times(0,\infty)$, then from the strong parabolic maximum principle with $c(x,t)\geqslant 0$ (Theorem 12 in \cite{evans}, Chapter 7), we can obtain that for $(x,t)\in\Sigma_\lambda\times(0,t_0]$, we have $\overline{z}(x,t)\equiv0$ and therefore              $z(x,t)=e^{c_0t}\,\overline{z}(x,t)\equiv0$.
\end{proof}
\begin{lem}[Parabolic hopf's lemma for not necessarily non-negative coefficient]
\label{parabolic_hopf_lemma}
For each $\lambda\in(0,1)$, we let $(x_0,t_0)$ be a point on the boundary of $\Sigma_\lambda\times(0,T)$ for $\forall
T>0$ such that $z(x_0,t_0)=0$ is the minimum in $\overline{\Sigma_\lambda}\times[0,T]$. Assume that there exists a neighborhood $V:=|x-x_0|^2+|t-t_0|^2<R^2_0$ of $(x_0,t_0)$ such that for $(x,t)\in V\cap\left(\Sigma_\lambda\times(0,T)\right)$, $z(x,t)>0$ and satisfies
\begin{equation*}
\frac{\partial z(x,t)}{\partial t}-\Delta z(x,t)-c(x,t)z(x,t)\geqslant 0,
\end{equation*} 
where $c(x,t)$ is bounded in $\Sigma_\lambda\times(0,T)$. If there exists a sphere $S:=|x-x'|^2+|t-t'|^2<R$ passing through $(x_0,t_0)$ and contained in $\overline{\Sigma_\lambda}\times[0,T]$ and $(x_0,t_0)\not=(x',t')$ , then under the assumptions made above, we have 
\begin{equation*}
\frac{\partial z}{\partial n}(x_0,t_0)<0,
\end{equation*}
where $n$ is the unit outer normal of $\partial \Sigma_\lambda\,$ for the fixed $t_0$. 
\end{lem}
\begin{proof}
For each $\lambda\in(0,1)$, we set $c_0=\mathop{sup}\limits_{\Sigma_\lambda\times(0,\infty)}c(x,t)$ and let
\begin{equation*}
\overline{z}(x,t)=e^{-c_0t}z(x,t),
\end{equation*}
then at the point $(x_0,t_0)$, $z(x,t)$ also gets its minimum $0$ and for $(x,t)\in V\cap\left(\Sigma_\lambda\times(0,T)\right)$, $\overline{z}(x,t)>0$ and satisfies that
\begin{equation*}
\frac{\partial\overline{z}(x,t)}{\partial t}-\Delta\overline{z}(x,t)+(c_0-c(x,t))\overline{z}(x,t)\geqslant 0.
\end{equation*}
Considering $(c_0-c(x,t))\geqslant 0$ in $\sigma_\lambda\times(0,T)$, then from the parabolic hopf's lemma(Theorem 2 in \cite{friedman-1958}), we can derive that every outer non-tangential derivative $\partial \overline{z}\backslash \partial\tau$ at $(x_0,t_0)$ is negative. Particularly, by the definition of $\overline{z}(x,t)$, for the fixed $t_0$ and the unit outer normal $n$ of $\partial \Sigma_\lambda$, we have $\frac{\partial z}{\partial n}(x_0,t_0)=\frac{\partial \overline{z}}{\partial n}(x_0,t_0)<0$.
\end{proof}

\section{Proof of the Main Theorem}\label{sec:proof}
We will carry out the proof in two steps. For simplicity, choose any direction within the region to be the $x_1$ direction. The first step is to show that for $\lambda$ sufficiently close to the right end of the domain, the following holds for all $\varphi\in\omega(u)$
\begin{equation}\label{The_starting_position}
\psi(x)\geqslant0,\quad\forall x\in\Sigma_\lambda.
\end{equation}
This provides the initial position to move the plane. We then move the plane $T_\lambda$ to the left as long as the inequality \eqref{The_starting_position} continues to hold, until it reaches its limiting position.
Define
\begin{equation}\label{lambda0_definition}
\lambda_0=\inf\{\lambda\geqslant0\,|\,\psi_\mu(x)\geqslant0, \,\textrm{for all} \,\,\varphi\in\omega(u),\, x\in\Sigma_\mu,\,\mu\geqslant\lambda\}.
\end{equation}
We will show that $\lambda_0=0$. Since $x_1$ direction can be chosen arbitrarily,  this implies that for any $\varphi\in\omega(u)$, $\varphi(x)$ is radially symmetric and $(1+|x|^2)^{\frac{N-2}{2}}\varphi(x)$ is monotone decreasing about the origin. We will now detail these two steps.

\begin{proof}[Proof of Theorem \ref{main_theorem}.] For all $\varphi\in\omega(u)$, we assume $\varphi\not\equiv0$ in $B_1(0)$.

\textbf{Step 1.} We show that for $\lambda<1$ and sufficiently close to $1$, the following holds:
\begin{equation}\label{result_step1}
\psi_\lambda(x)\geqslant0, \quad \forall x\in\Sigma_\lambda, \quad \forall \varphi\in\omega(u).
\end{equation}

The Lipschitz continuity assumption (F1) on $f$ implies that $c_\lambda(x,t)$ is bounded. 
Additionally, we have
\begin{equation*}
z_\lambda(x,t)=\left(\frac{1+|x^\lambda|^2}{1+|x|^2}\right)^\frac{N-2}{2}u(x^\lambda,t)-u(x,t)>0,\quad(x,t)\in\partial\Sigma_\lambda\times(0,\infty)    
\end{equation*}
since $u(x,t)=0$, for $(x,t)\in \partial B_1(0)\times(0,\infty)$ and $u$ is positive in $B_1(0)$. Combining this with \eqref{zlambda_inequality}, we can apply Lemma \ref{asymptotic_narrow_region_principle} to conclude that \eqref{result_step1} holds.

\textbf{Step 2.} We will demonstrate that the parameter $\lambda_0$, as defined in \eqref{lambda0_definition}, is equal to zero, that is to say
\begin{equation}\label{stopping_position}
\lambda_0=0.
\end{equation}

Assume for contradiction that $\lambda_0>0$. We will demonstrate that $T_{\lambda_0}$ can be shifted slightly to the left, thereby contradicting the definition of $\lambda_0$.

To begin with, we intend to determine the sign of $\psi_{\lambda_0}$ for any $\varphi\in\omega(u)$ and all $x\in\Sigma_{\lambda_0}$ under the case $\lambda_0>0$. To achieve this, according to the definition \eqref{psilambda_definition}, we now need to discuss the inequality that $z_{\lambda_0}(x,t)$ satisfies as $t\to\infty$. From the definition of $\omega(u)$, for each $\varphi\in \omega(u)$, there exists a sequence $\{t_k\}$ such that $u(x,t_k)\to \varphi(x)$ as $t_k\to \infty$.
Define
\begin{equation}\label{uk_definition}
    u_k(x,t)=u(x,t+t_k-1),
\end{equation}
and
\begin{equation}\label{fk_definition}
f_k(|x|,u,t)=f(|x|,u,t+t_k-1).
\end{equation}
Then we have $u_k(x,1)\to\varphi(x)$ in the sense of $C(B_1(0))$ as $k\to\infty$.
Let $Q_1:=B_1(0)\times[1-\varepsilon_0,1+\varepsilon_0]$. 
We now have
\begin{equation}\label{uk_problem}
\left\{\begin{array}{ll}
\frac{\partial u_k}{\partial t}(x,t)-\Delta u_k(x,t)=f_k(|x|,u_k(x,t),t), &(x,t)\in Q_1,  \\
u_k(x,t)=0, & (x,t)\in \partial B_1(0)\times[1-\varepsilon_0,1+\varepsilon_0].
\end{array}\right .
\end{equation}

Similarly as \eqref{v_definition}-\eqref{zlambda_ulambda_u}, we have the following definitions of functions
\begin{eqnarray}
	\label{vk_definition}
	v_k(x,t)&=&(1+|x|^2)^\frac{N-2}{2}u_k(x,t),\\
        \label{wk_definition}
	w_{\lambda_0,k}(x,t)&=&v_k(x^{\lambda_0},t)-v_k(x,t),\\
	\label{zlambda0k_definition}	z_{{\lambda_0},k}(x,t)&=&z_{\lambda_0}(x,t+t_k-1)=\left(\frac{1+|x^{\lambda_0}|^2}{1+|x|^2}\right)^\frac{N-2}{2}u_k(x^{\lambda_0},t)-u_k(x,t).
\end{eqnarray}
It follows from \eqref{uk_problem} and Lemma \ref{lemma_zlambda_inequality} that
\begin{equation*}
\frac{\partial z_{{\lambda_0},k}}{\partial t}(x,t)-\Delta z_{{\lambda_0},k}(x,t)\geqslant c_{{\lambda_0},k}(x,t)z_{{\lambda_0},k}(x,t),\quad(x,t)\in\Sigma_{\lambda_0}\times[\overline{t},\infty),
\end{equation*}
where
\begin{equation*}
c_{{\lambda_0},k}(x,t)=c_{\lambda_0}(x,t+t_k-1)=\frac{f_k(|x|,(1+|x|^2)^{-\frac{N-2}{2}}v_k(x^{\lambda_0},t),t)-f_k(|x|,u_k(x,t),t)}{(1+|x|^2)^{-\frac{N-2}{2}}v_k(x^{\lambda_0},t)-(1+|x|^2)^{-\frac{N-2}{2}}v_k(x,t)}.
\end{equation*}

 Based on the relationship between the functions \eqref{vk_definition}-\eqref{zlambda0k_definition} and $u_k$ defined as \eqref{uk_definition}, using the Lemma \ref{Convergence_of_sequences}, we deduce the existence of subsequences $v_k$, $w_{\lambda_0,k}$ and $z_{{\lambda_0},k}$ which converge uniformly to the respective functions $v_\infty$, $w_{\lambda_0,\infty}$ and $z_{{\lambda_0},\infty}$ all in the sense of $C^{2,1}\left(\Sigma_{\lambda_0}\times[1-\varepsilon_0,1+\varepsilon_0]\right)$ as $k \to \infty$ and they satisfies
 \begin{equation*}
 \begin{aligned}
 z_{\lambda_0,\infty}(x,t)&=(1+|x|^2)^{-\frac{N-2}{2}}w_{\lambda_0,\infty}(x,t)=(1+|x|^2)^{-\frac{N-2}{2}}\left(v_\infty(x^{\lambda_0},t)-v_\infty(x,t)\right)\\&=\left(\frac{1+|x^{\lambda_0}|^2}{1+|x|^2}\right)^\frac{N-2}{2}u_\infty(x^{\lambda_0},t)-u_\infty(x,t).
 \end{aligned}
 \end{equation*}
 Furthermore, both in the sense of $C\left(\Sigma_{\lambda_0}\times[1-\varepsilon_0,1+\varepsilon_0]\right)$ as $k\to\infty$ we have
\begin{eqnarray*}
\frac{\partial z_{{\lambda_0},k}}{\partial t}(x,t)-\Delta z_{{\lambda_0},k}(x,t)&\to&\frac{\partial z_{{\lambda_0},\infty}}{\partial t}(x,t)-\Delta z_{{\lambda_0},\infty}(x,t),\\ c_{{\lambda_0},k}(x,t)&\to& c_{{\lambda_0},\infty}(x,t),
\end{eqnarray*}
where $c_{{\lambda_0},\infty}(x,t)$ is bounded in $\Sigma_{\lambda_0}\times[1-\varepsilon_0,1+\varepsilon_0]$, this follows from the proof of Lemma \ref{Convergence_of_sequences}, where it is established that the sequence $f_k$ within the definition of $c_{\lambda_0,k}$ uniformly converges to $f_\infty$ satisfying (F1) in the sense of $C\left(\Sigma_{\lambda_0}\times[1-\varepsilon_0,1+\varepsilon_0]\right)$. For any $\varphi\in\omega(u)$, by the definition of the limit set $\omega(u)$, there exists $t_k$ such that $z_{\lambda_0}(x,t_k)\to\psi_{\lambda_0}(x)$ in the sense of $C(\Sigma_{\lambda_0})$ as $t_k\to\infty$ . Particularly, according to the convergence of $z_{\lambda_0,k}$, in the sense of $C^2(\Sigma_{\lambda_0})$ we have
\begin{equation}\label{zlambdainftyx1_equals_psi}
z_{\lambda_0}(x,t_k)=z_{{\lambda_0},k}(x,1)\to z_{{\lambda_0},\infty}(x,1)=\psi_{\lambda_0}(x) \quad as\,\,k\to\infty.
\end{equation}

Combining with the continuity of $z_{\lambda_0,\infty}$ with respect to $t$ and the definition of $\lambda_0$, we deduce $z_{{\lambda_0},\infty}(x,t)$ satisfies the following inequalities
\begin{equation}\label{problem_strong_parabolic_maximum_principle}
\left\{
\begin{aligned}
&\frac{\partial z_{{\lambda_0},\infty}(x,t)}{\partial t}-\Delta z_{{\lambda_0},\infty}(x,t)\geqslant c_{{\lambda_0},\infty}(x,t)z_{{\lambda_0},\infty}(x,t),\quad &(x,t)\in\Sigma_{\lambda_0}\times[1-\varepsilon_0,1+\varepsilon_0],  \\
&z_{{\lambda_0},\infty}(x,t)\geqslant0, \quad &(x,t)\in\Sigma_{\lambda_0}\times[1-\varepsilon_0,1+\varepsilon_0]. 
\end{aligned}\right.                               
\end{equation}
We apply the Lemma \ref{strong_parabolic_maximum_principle} to \eqref{problem_strong_parabolic_maximum_principle} in $\Sigma_{{\lambda_0}}\times[1-\varepsilon_0,1+\varepsilon_0]$ to get either
\begin{equation}\label{zlambda0infty_eqiv_0}
z_{{\lambda_0},\infty}(x,t)\equiv0\quad \textrm{for} \quad (x,t)\in\Sigma_{\lambda_0}\times[1-\varepsilon_0,1+\varepsilon_0]
\end{equation}
or
\begin{equation}\label{zlambda0infty_gg_0}
z_{{\lambda_0},\infty}(x,t)> 0\quad \textrm{for} \quad (x,t)\in\Sigma_{\lambda_0}\times(1-\varepsilon_0,1+\varepsilon_0].
\end{equation}
In case \eqref{zlambda0infty_eqiv_0}, since $z_{{\lambda_0},\infty}(x,t)\equiv0$, for all $(x,t)\in\Sigma_{\lambda_0}\times[1-\varepsilon_0,1+\varepsilon_0]$ we have 
\begin{equation}\label{vinfty(xlambda)_equal_vinfty(x)}
w_{\lambda_0,\infty}(x,t)\equiv0,\quad i.e.,\quad v_{\infty}(x^{\lambda_0},t)\equiv v_\infty(x,t).
\end{equation}
From \eqref{v_equation}, for $(x,t)\in\Sigma_{\lambda_0}\times(1-\varepsilon_0,1+\varepsilon_0]$ we have
\begin{equation}\label{vx_equation}
\Delta_{(g,x)}v_\infty(x,t)= \left(\frac{1+|x|^2}{2}\right)^{2}\frac{\partial v_\infty}{\partial t}(x,t)-\frac{(1+|x|^2)^{\frac{N+2}{2}}}{4}f_\infty(|x|,(1+|x|^2)^{-\frac{N-2}{2}}v_\infty(x,t),t)+\frac{N(N-2)}{4}v_\infty(x,t),
\end{equation}
and
\begin{equation}\label{vxlambda0_equation}
\begin{aligned}
\Delta_{(g,x)}v_\infty(x^{\lambda_0},t)&= \left(\frac{1+|x^{\lambda_0}|^2}{2}\right)^{2}\frac{\partial v_\infty}{\partial t}(x^{\lambda_0},t)\\&\quad-\frac{(1+|x^{\lambda_0}|^2)^{\frac{N+2}{2}}}{4}f_\infty(|x^{\lambda_0}|,(1+|x^{\lambda_0}|^2)^{-\frac{N-2}{2}}v_\infty(x^{\lambda_0},t),t)+\frac{N(N-2)}{4}v_\infty(x^{\lambda_0},t).
\end{aligned}
\end{equation}
By \eqref{vinfty(xlambda)_equal_vinfty(x)} we can obtain that
\begin{equation*}
\Delta_{(g,x)}w_{\lambda_0,\infty}(x,t)\equiv0,\quad\left(\frac{1+|x^\lambda|^2}{2}\right)^{2}\frac{\partial w_{\lambda_0,\infty}}{\partial t}(x^\lambda,t)\equiv0,\quad \frac{N(N-2)}{4}w_{\lambda_0,\infty}(x,t)\equiv0
\end{equation*}
for $(x,t)\in\Sigma_{\lambda_0}\times[1-\varepsilon_0,1+\varepsilon_0]$. Combining \eqref{vx_equation} and \eqref{vxlambda0_equation}, we finally arrive at
\begin{equation}\label{fx_equals_fxlambda}
(1+|x|^2)^{\frac{N+2}{2}}f_\infty(|x|,(1+|x|^2)^{-\frac{N-2}{2}}v_\infty(x,t),t)\equiv(1+|x^{\lambda_0}|^2)^{\frac{N+2}{2}}f_\infty(|x^{\lambda_0}|,(1+|x^{\lambda_0}|^2)^{-\frac{N-2}{2}}v_\infty(x,t),t)
\end{equation}
for $(x,t)\in\Sigma_{\lambda_0}\times(1-\varepsilon_0,1+\varepsilon_0]$. Since Lemma \ref{proof_of_xlambda_distance} implies that $|x|>|x^{\lambda_0}|$, \eqref{fx_equals_fxlambda} contradicts the assumption (F3) of $f$. It follows that case \eqref{zlambda0infty_eqiv_0} is invalid.

Next from case \eqref{zlambda0infty_gg_0}, we can derive that for all $\varphi\in\omega(u)$
\begin{equation}\label{psilambda0_geq_0}
z_{{\lambda_0},\infty}(x,1)=\psi_{\lambda_0}(x)>0,\quad x\in\Sigma_{\lambda_0}.
\end{equation}
Since $\lambda_0>0$, we now attempt to slightly shift $\lambda_0$ to the left, if $\lambda_0$ still meets definition \eqref{lambda0_definition}, we can construct a contradiction. For any $l>0$ small, we set $\overline{V_{\lambda_0+l}}=x\in\left\{x\in\overline{B_1(0)}\,\left||x-e_{\lambda_0}|\geqslant\frac{1+\lambda_0^2}{2\lambda_0}+l\right.\right\}$, for $\psi_{\lambda_0}$, there exists a $C_\varphi>0$, such that
\begin{equation}\label{psilambda0_geq_cphi}
\psi_{\lambda_0}\geqslant C_\varphi>0,\quad x\in\overline{V_{\lambda_0+l}}.
\end{equation}
We now show that, for all $\varphi\in\omega(u)$, there exists an universal constant $C_0$ such that
\begin{equation}\label{psilambda0_geq_c0}
\psi_{\lambda_0}\geqslant C_0>0,\quad x\in\overline{V_{\lambda_0+l}}.
\end{equation}
If not, there exists a sequence of functions $\{\psi_{\lambda_0}^k\}$ with respect to ${\varphi^k}\subset\omega(u)$ and a sequence of points $\{x^k\}\subset\overline{V_{\lambda_0+l}}$ such that for each $k$ we have 
\begin{equation}\label{psilambda0k_leq_1/k}
\psi_{\lambda_0}^k(x^k)<\frac{1}{k}.
\end{equation}
By the compactness of $\omega(u)$ in $C(\overline{B_1(0)})$, there exists $\psi_{\lambda_0}^0$ which corresponds to some $\varphi^0\in\omega(u)$ and $x^0\in\overline{V_{\lambda_0+l}}$  such that
\begin{equation*}
\psi_{\lambda_0}^k(x^k)\to\psi_{\lambda_0}^0(x^0),
\end{equation*}
as $k\to\infty$ in the sense of $C(B_1(0))$. Now by \eqref{psilambda0k_leq_1/k} and the definition of $\lambda_0$, we obtain
\begin{equation*}
\psi_{\lambda_0}^0(x^0)=0,
\end{equation*}
which contradicts \eqref{psilambda0_geq_0}, since $\varphi^0\in\omega(u)$. Thus \eqref{psilambda0_geq_c0} must be established.

From \eqref{psilambda0_geq_c0} and the continuity of $\psi_\lambda$ with respect to $\lambda$, for each $\psi_\lambda$, under a fixed $C_0$, there exist $\varepsilon_\varphi>0$ such that
\begin{equation*}
\psi_\lambda(x)\geqslant\frac{C_0}{2}>0,\quad x\in\overline{V_{\lambda_0+l}},\quad\forall\lambda\in(\lambda_0-{\varepsilon_\varphi}\,,\lambda_0).
\end{equation*}
Similarly due to the compactness of $\omega(u)$ in $C(\overline{B_1(0)})$, for all $\psi_\lambda$, there exists a universal $\varepsilon>0$ such that
\begin{equation}\label{psilambda_geq_c0/2}
\psi_\lambda(x)\geqslant\frac{C_0}{2}>0\,,\quad x\in\overline{V_{\lambda_0+l}}\,,\quad\forall\lambda\in(\lambda_0-{\varepsilon}\,,\lambda_0).
\end{equation}
Consequently, for $t$ sufficiently large, we have
\begin{equation*}
z_\lambda(x,t)\geqslant0\,,\quad x\in\overline{V_{\lambda_0+l}}\,,\quad\forall\lambda\in(\lambda_0-{\varepsilon}\,,\lambda_0).
\end{equation*}
Since $l>0$ is small, by \eqref{psilambda0k_leq_1/k}, we can choose $\varepsilon>0$ small, such that $\Sigma_\lambda\backslash V_{\lambda_0+l}$ is a narrow region defined as \eqref{narrow_annulus_definition} for $\lambda\in(\lambda_0-\varepsilon,\lambda_0)$, then applying the \emph{asymptotic narrow region principle} (Lemma \ref{asymptotic_narrow_region_principle}), we arrive at 
\begin{equation}\label{psilambda_geq_0_in_narrowregion}
\psi_\lambda(x)\geqslant0\,,\quad\forall x\in\Sigma_\lambda\backslash V_{\lambda_0+l}.
\end{equation}
Combining \eqref{psilambda_geq_c0/2} and \eqref{psilambda_geq_0_in_narrowregion}, we derive that
\begin{equation*}
\psi_\lambda(x)\geqslant0\,,\quad\forall x\in\Sigma_\lambda\,,\quad\forall\lambda\in(\lambda_0-{\varepsilon}\,,\lambda_0)\,,\quad\forall\varphi\in\omega(u).
\end{equation*}
This contradicts the definition of $\lambda_0$. Therefore, $\lambda_0=0$ must be true.

As a result, \eqref{stopping_position} implies that for all $\varphi\in\omega(u)$
\begin{equation*}
\psi_0(x)\geqslant0\,,\quad\forall x\in\Sigma_0,
\end{equation*}
that is to say, for all $\varphi\in\omega(u)$ and $x\in\Sigma_0$
\begin{equation*}
\begin{aligned}
z_0(x,1)&=\left(\frac{1+|x^0|^2}{1+|x|^2}\right)^\frac{N-2}{2}u_\infty(x^0,1)-u_\infty(x,1)\\&=\left(\frac{1+|x^0|^2}{1+|x|^2}\right)^\frac{N-2}{2}\varphi(x^0)-\varphi(x)\geqslant0.
\end{aligned}
\end{equation*}
Since $x^0=(-x_1,x_2,\cdots,x_N)$ for $x\in\Sigma_0$ with $0<x_1<1$, finally we can get that 
\begin{equation}\label{phi-x1_geq_phix1}
\varphi(-x_1,x_2,\cdots,x_N)\geqslant\varphi(x_1,x_2,\cdots,x_N)
\end{equation}
Since the $x_1$-direction can be chosen arbitrarily, \eqref{phi-x1_geq_phix1} implies that all $\varphi(x)$ are radially symmetric about the origin.
Combining with \eqref{problem_strong_parabolic_maximum_principle} and the proof of \eqref{psilambda0_geq_0}, we can derive that for $0<\lambda<1$, $z_{\lambda,\infty}$ satisfies 
\begin{equation*}
\left\{
\begin{aligned}
&\frac{\partial z_{\lambda,\infty}(x,t)}{\partial t}-\Delta z_{\lambda,\infty}(x,t)\geqslant c_{\lambda,\infty}(x,t)z_{\lambda,\infty}(x,t),\quad &(x,t)\in\Sigma_{\lambda}\times[1-\varepsilon_0,1+\varepsilon_0],  \\
&z_{\lambda,\infty}(x,t)=0, \quad &(x,t)\in T_\lambda \times[1-\varepsilon_0,1+\varepsilon_0],\\&z_{\lambda,\infty}(x,t)>0,&(x,t)\in\Sigma_{\lambda}\times[1-\varepsilon_0,1+\varepsilon_0]. 
\end{aligned}\right.                               
\end{equation*}
where $c_{\lambda,\infty}$ is bounded, then we can apply Lemma \ref{parabolic_hopf_lemma} to obtain that 
\begin{equation*}
\frac{\partial z_{\lambda,\infty}(x,1)}{\partial n}=\frac{\partial \psi_\lambda(x)}{\partial n}<0,\quad x\in T_\lambda,\quad \forall\,0<\lambda<1,
\end{equation*}
where $n$ is the outer normal vector of $\partial\overline{\Sigma_\lambda}$. From Lemma \ref{lemma_unit_outer_normal}, we can conclude that
\begin{equation}\label{partialphipartialn_gep_0}
\frac{\partial[(1+|x|^2)^{\frac{N-2}{2}}\varphi(x)]}{\partial n}>0.
\end{equation}
Under the conclusion that all $\varphi(x)$ are radially symmetric about the origin, from \eqref{partialphipartialn_gep_0} we can infer that for all $0<r<1$
\begin{equation}
\partial_r \left((1+r^2)^{\frac{N-2}{2}}\varphi(r)\right)<0,
\end{equation}
which shows the asymptotic monotonicity, now we complete the proof of Theorem \ref{main_theorem}.
\end{proof}
\section{Appendix}

We note that $B_1(0)=\{x\in\mathbb{R}^N:|x|<1\}$, $N\geqslant3$, set $S^+=\{X\in\mathbb{R}^{N+1}:|X|=1,X_{N+1}>0\}$ and define a mapping $P:(S^+,|dX|^2)\to(B_1(0),g)$ by 
\begin{equation*}
P(X_1,\cdots,X_N,X_{N+1})=\frac{1}{X_{N+1}+1}(X_1,\cdots,X_N).
\end{equation*}
The overall idea of the proof of \eqref{Tlambda_definition}-\eqref{xlambda_definition} is similar to that in \cite{shioji_watanabe-2012} in the case $a=1$, we omit the proof process. Strongly inspired by the Lemma A.1 and Lemma A.2 in \cite{naito_suzuki-1998}, here we give the proof of \eqref{xlambda_distance} and \eqref{operator_identity}.
\begin{lem}\label{proof_of_xlambda_distance}
We have
\begin{equation}\label{1+xlambda_comparison}
\frac{1+|x|^2}{1+|x^\lambda|^2}=\left(\frac{2\lambda}{1+\lambda^2}\right)^2|x-e_\lambda|^2.
\end{equation}
Therefore, we can derive $|x|>|x^\lambda|$  for $x\in\Sigma_\lambda$.
\end{lem}

\begin{proof}
We define $a=-e_\lambda /|e_\lambda|^2=(2\lambda/(1-{\lambda}^2),0,\cdots,0)$. From this definition, \eqref{elambda_definition} and \eqref{xlambda_definition}, by some elementary computations, we note that
\begin{eqnarray*}
a^\lambda &=& e_\lambda+\left(\frac{1+\lambda^2}{2\lambda}\right)^2\frac{a-e_\lambda}{|a-e_\lambda|^2}=0,
\\ |e_\lambda|^2+1&=&\left(\frac{1+\lambda^2}{2\lambda}\right)^2 = \frac{1-\lambda^2}{2\lambda}\cdot\frac{(1+\lambda^2)^2}{2\lambda(1-\lambda)^2}=|e_\lambda||a-e_\lambda|.
\end{eqnarray*}
Then we have 
\begin{eqnarray*}
|x^\lambda|^2 &=& |x^\lambda-a^\lambda|^2 \\
&=&\left|\left(\frac{1+\lambda^2}{2\lambda}\right)^2\left(\frac{x-e_\lambda}{|x-e_\lambda|^2}-\frac{a-e_\lambda}{|a-e_\lambda|^2}\right)\right|^2=(|e_\lambda|^2+1)^2\left|\frac{x-e_\lambda}{|x-e_\lambda|^2}-\frac{a-e_\lambda}{|a-e_\lambda|^2}\right|^2 \\
&=& \frac{(|e_\lambda|^2+1)^2|x-a|^2}{|x-e_\lambda|^2|a-e_\lambda|^2}=\frac{|e_\lambda|^2|a-e_\lambda|^2|x-a|^2}{|x-e_\lambda|^2|a-e_\lambda|^2}=\frac{|e_\lambda|^2|x-a|^2}{|x-e_\lambda|^2},
\end{eqnarray*}
where we use the property of distance operations that for $p,q\in\mathbb{R}^N\backslash\{0\}$
\begin{equation*}
\left|\frac{p}{|p|^2}-\frac{q}{|q|^2}\right|=\frac{|p-q|}{|p||q|}.
\end{equation*}
Thus we obtain
\begin{equation*}
1+|x^\lambda|^2=1+\frac{|e_\lambda|^2|x-a|^2}{|x-e_\lambda|^2}=1+\frac{|e_\lambda|^2|x+e_\lambda/|e_\lambda|^2|^2}{|x-e_\lambda|^2}=\frac{(|e_\lambda|^2+1)(1+|x|^2)}{|x-e_\lambda|^2}=\left(\frac{1+\lambda^2}{2\lambda}\right)^2\frac{(1+|x|^2)}{|x-e_\lambda|^2},
\end{equation*}
which implies \eqref{1+xlambda_comparison}. Due to the definition of $\Sigma_\lambda$, for $x\in\Sigma_\lambda$, we have
\begin{equation*}
\frac{1+|x|^2}{1+|x^\lambda|^2}=\left(\frac{2\lambda}{1+\lambda^2}\right)^2|x-e_\lambda|^2>\left(\frac{2\lambda}{1+\lambda^2}\right)^2\left(\frac{1+\lambda^2}{2\lambda}\right)^2=1,
\end{equation*}
then we can find that $|x|>|x^\lambda|$.
\end{proof}

\begin{lem}\label{proof_of_operater_identity}
Assume that $v(x)\in C^{2}(B_1(0))$. Let $y=x^\lambda$ and $v(y)$ is a function with $y$ as the independent variable. We define $v_\lambda(x)=v(x^\lambda)$ as a new form of function with $x$ as the independent variable. Then  
\begin{equation}\label{detailed_operater_identity}
\begin{aligned}
\Delta_{(g,x)}v_\lambda(x) &= \left(\frac{1+|x|^2}{2}\right)^2\left(\Delta_{x}v_\lambda(x)-\frac{2(N-2)}{1+|x|^2}x\cdot\nabla_{x}v_\lambda(x)\right) \\
&= \left.\left(\frac{1+|y|^2}{2}\right)^2\left(\Delta_{y}v(y)-\frac{2(N-2)}{1+|y|^2}y\cdot\nabla_{y}v(y)\right)\right|_{y=x^\lambda}=\Delta_{(g,y)}v(y)|_{y=x^\lambda},
\end{aligned}
\end{equation}            
where $\Delta_x=\sum_{i=1}^{N}\partial^2/\partial x_i^2$,\quad$x\cdot\nabla_x=\sum_{i=1}^{N}x_i\partial/\partial x_i$,\,\, and $\Delta_{(g,x)}$\,\,\,defined\,\,as\, \eqref{operatordefinition}.
\end{lem}
\begin{proof}
To compare the values of $\Delta_{(g,x)}v_\lambda(x)$ and $\Delta_{(g,y)}v(y)$ through direct computation, we define $u(y)$ and $u_\lambda(x)$ as
\begin{eqnarray*}
u(y)&=&(1+|y|^2)^{-(N-2)/2}v(y),\\
u_\lambda(x)&=&(1+|x|^2)^{-(N-2)/2}v_\lambda(x).
\end{eqnarray*}
Then we can find that
\begin{eqnarray}
\label{deltayu_definition}\frac{1}{4}(1+|y|^2)^{(N+2)/2}\Delta_{y}u(y)&=&\left(\frac{1+|y|^2}{2}\right)^2\left(\Delta_{y}v-\frac{2(N-2)}{1+|y|^2}y\cdot\nabla_{y}v\right)-\frac{N(N-2)}{4}v,\\ \label{deltaxulambda_definition}
\qquad\quad\frac{1}{4}(1+|x|^2)^{(N+2)/2}\Delta_{x}u_\lambda(x)&=&\left(\frac{1+|x|^2}{2}\right)^2\left(\Delta_{x}v_\lambda-\frac{2(N-2)}{1+|x|^2}x\cdot\nabla_{x}v_\lambda\right)-\frac{N(N-2)}{4}v_\lambda,\\ \label{uy_ulambda}
u(y)&=&\left(\frac{1+|x|^2}{1+|y|^2}\right)^{(N-2)/2}u_\lambda(x).
\end{eqnarray}
For simplicity, we define 
\begin{equation}\label{XYU_definition}
X=x-e_\lambda,\quad Y=y-e_\lambda, \quad U_\lambda(X)=u_\lambda(x)\,\,\,and\,\,\,U(Y)=u(y).
\end{equation} 
By \eqref{xlambda_definition}, \eqref{1+xlambda_comparison} and \eqref{uy_ulambda}, it follows that
\begin{eqnarray}\label{Y_definition}
Y=\left(\frac{1+\lambda^2}{2\lambda}\right)^2\frac{X}{|X|^2},\quad X=\left(\frac{1+\lambda^2}{2\lambda}\right)^2\frac{Y}{|Y|^2}\\
U(Y,t)= \left(\frac{2\lambda}{1+\lambda^2}\right)^{N-2}|X|^{N-2}U_\lambda(X).
\end{eqnarray}
Then we can calculate that
\begin{eqnarray*}
\left.\frac{\partial X_j}{\partial Y_i}\right |_{i=j}&=&\left(\frac{1+\lambda^2}{2\lambda}\right)^{2}\frac{|Y|^2-2Y_i^2}{|Y|^4},\quad\left.\frac{\partial X_j}{\partial Y_i}\right|_{i\not=j}=\quad\left(\frac{1+\lambda^2}{2\lambda}\right)^{2}\frac{-2Y_i Y_j}{|Y|^4},\\
\frac{\partial U(Y)}{\partial Y_i}&=&\sum_{j=1}^{N}\frac{\partial U(Y)}{\partial X_j}\frac{\partial X_j}{\partial Y_i}\\&=&\left(\frac{2\lambda}{1+\lambda^2}\right)^{N-2}\sum_{j=1}^{N}\left(\frac{\partial|X|^{N-2}}{\partial X_j}U_\lambda(X)+\frac{\partial U_\lambda(X)}{\partial X_j}|X|^{N-2}\right)\frac{\partial X_j}{\partial Y_i},\\
\frac{\partial^2 U(Y)}{\partial Y_i^2}&=&\sum_{j=1}^{N}\left(\frac{\partial^2 U(Y)}{\partial X_j^2}\left(\frac{\partial X_j}{\partial Y_i}\right)^2+\frac{\partial U(Y)}{\partial X_j}\frac{\partial^2 X_j}{\partial Y_i^2}\right)+\sum_{1\leqslant p\not=q}^{N}\frac{\partial^2U(Y)}{\partial X_p\partial X_q}\frac{\partial X_p}{\partial Y_i}\frac{\partial X_q}{\partial Y_i}\\ \frac{\partial^2U(Y)}{\partial X_j^2}&=&\left(\frac{2\lambda}{1+\lambda^2}\right)^{N-2}\left(\frac{\partial^2|X|^{N-2}}{\partial X_j^2}U_\lambda(X)+2\frac{\partial |X|^{N-2}}{\partial X_1}\frac{\partial U_\lambda(X)}{\partial X_j}+|X|^{N-2}\frac{\partial^2U_\lambda(X)}{\partial X_j^2}\right)\\\frac{\partial^2U(Y)}{\partial X_pX_q}&=&\left(\frac{2\lambda}{1+\lambda^2}\right)^{N-2}\left(\frac{\partial^2|X|^{N-2}}{\partial X_p\partial X_q}U_\lambda(X)+|X|^{N-2}\frac{\partial U_\lambda(X)}{\partial X_p \partial X_q}\right)\\&\quad&+\left(\frac{2\lambda}{1+\lambda^2}\right)^{N-2}\left(\frac{\partial|X|^{N-2}}{\partial X_p}\frac{\partial U_\lambda(X)}{\partial X_q}+\frac{\partial|X|^{N-2}}{\partial X_q}\frac{\partial U_\lambda(X)}{\partial X_p}\right).
\end{eqnarray*}
Then we obtain
\begin{eqnarray*}            &\quad&\Delta_YU(Y)=\sum_{i=1}^{N}\frac{\partial^2 U(Y)}{\partial Y_i^2}\\&=&\left(\frac{2\lambda}{1+\lambda^2}\right)^{N-2}|X|^{N-2}\left(\sum_{j=1}^{N}\frac{\partial^2U_\lambda(X)}{\partial X_j^2}\sum_{i=1}^{N}\left(\frac{\partial X_j}{\partial Y_i}\right)^2+\sum_{1\leqslant p\not=q}^{N}\frac{\partial^2U_\lambda(X)}{\partial X_p \partial X_q}\sum_{i=1}^{N}\frac{\partial X_p}{\partial Y_i}\frac{\partial X_q}{\partial Y_i}\right)\\&\quad&+\left(\frac{2\lambda}{1+\lambda^2}\right)^{N-2}\sum_{j=1}^{N}\frac{\partial U_\lambda(X)}{\partial X_j}\left(2\frac{\partial |X|^{N-2}}{\partial X_j}\sum_{i=1}^{N}\left(\frac{\partial X_j}{\partial Y_i}\right)^2+|X|^{N-2}\sum_{i=1}^{N}\frac{\partial^2 X_j}{\partial Y_i^2}\right)\\&\quad&+\left(\frac{2\lambda}{1+\lambda^2}\right)^{N-2}\sum_{j=1}^{N}\frac{\partial U_\lambda(X)}{\partial X_j}\left(\sum_{i\leqslant m\not= j}^{N}\frac{\partial|X|^{N-2}}{\partial X_m}\sum_{i=1}^{N}\frac{\partial X_j}{\partial Y_i}\frac{\partial X_m}{\partial Y_i}\right)\\&\quad&+\left(\frac{2\lambda}{1+\lambda^2}\right)^{N-2}U_\lambda(X)\left(\sum_{j=1}^{N}\frac{\partial^2|X|^{N-2}}{\partial X_j^2}\sum_{i=1}^{N}\left(\frac{\partial X_j}{\partial Y_i}\right)^2+\sum_{1\leqslant p\not=q}^{N}\frac{\partial^2|X|^{N-2}}{\partial X_p \partial X_q}\sum_{i=1}^{N}\frac{\partial X_p}{\partial Y_i}\frac{\partial X_q}{\partial Y_i}\right)\\&\quad&+\left(\frac{2\lambda}{1+\lambda^2}\right)^{N-2}U_\lambda(X)\left(\sum_{j=1}^{N}\frac{\partial|X|^{N-2}}{\partial X_j}\sum_{i=1}^{N}\frac{\partial^2 X_j}{\partial Y_i^2}\right).
\end{eqnarray*}
From \eqref{Y_definition} we also have
\begin{equation*}
\begin{aligned}
&\sum_{i=1}^{N}\left(\frac{\partial X_j}{\partial Y_i}\right)^2=\left(\frac{1+\lambda^2}{2\lambda}\right)^{4}\frac{1}{|Y|^4}=\left(\frac{1+\lambda^2}{2\lambda}\right)^{4}\left(\frac{2\lambda}{1+\lambda^2}\right)^{8}|X|^4=\left(\frac{2\lambda}{1+\lambda^2}\right)^{4}|X|^4,\quad \sum_{i=1}^{N}\frac{\partial X_p}{\partial Y_i}\frac{\partial X_q}{\partial Y_i}=0,\\&2\frac{\partial |X|^{N-2}}{\partial X_j}\sum_{i=1}^{N}\left(\frac{\partial X_j}{\partial Y_i}\right)^2+|X|^{N-2}\sum_{i=1}^{N}\frac{\partial^2 X_j}{\partial Y_i^2}=0,\,\sum_{j=1}^{N}\frac{\partial^2|X|^{N-2}}{\partial X_j^2}\sum_{i=1}^{N}\left(\frac{\partial X_j}{\partial Y_i}\right)^2+\sum_{j=1}^{N}\frac{\partial|X|^{N-2}}{\partial X_j}\sum_{i=1}^{N}\frac{\partial^2 X_j}{\partial Y_i^2}=0.
\end{aligned}
\end{equation*}
Through the above calculation process, we finally get that
\begin{equation}\label{U_equation}
\Delta_Y U(Y)=\left(\frac{2\lambda}{1+\lambda^2}\right)^{N-2}|X|^{N-2}\left(\sum_{j=1}^{N}\frac{\partial^2U_\lambda(X)}{\partial X_j^2}\sum_{i=1}^{N}\left(\frac{\partial X_j}{\partial Y_i}\right)^2\right)=\left(\frac{2\lambda}{1+\lambda^2}\right)^{N+2}|X|^{N+2}\Delta_X U_\lambda(X).
\end{equation}
By \eqref{XYU_definition}, for each $i$, we have
\begin{eqnarray*}
\frac{\partial U_\lambda(X)}{\partial X_i}=\frac{\partial u_\lambda(x)}{\partial x_i}\frac{\partial x_i}{\partial X_i}=\frac{\partial u_\lambda(x)}{\partial x_i},&\,&\frac{\partial^2U_\lambda(X)}{\partial X_i^2}=\frac{\partial\left(\frac{\partial U_\lambda(X)}{\partial X_i}\right)}{\partial X_i}=\frac{\partial\left(\frac{\partial u_\lambda(x)}{\partial x_i}\right)}{\partial x_i}\frac{\partial x_i}{\partial X_i}=\frac{\partial^2 u_\lambda(x)}{\partial x_i^2}\\\frac{\partial U(Y)}{\partial Y_i}=\frac{\partial u(y)}{\partial y_i}\frac{\partial y_i}{\partial Y_i}=\frac{\partial u(y)}{\partial y_i} ,&\,&\frac{\partial^2U(U)}{\partial Y_i^2}=\frac{\partial\left(\frac{\partial U(Y)}{\partial Y_i}\right)}{\partial Y_i}=\frac{\partial\left(\frac{\partial u(y)}{\partial y_i}\right)}{\partial y_i}\frac{\partial y_i}{\partial Y_i}=\frac{\partial^2 u(y)}{\partial y_i^2}.
\end{eqnarray*}
Then \eqref{U_equation} implies 
\begin{equation*}
\Delta_y u=\left(\frac{2\lambda}{1+\lambda^2}\right)^{N+2}|x-e_\lambda|^{N+2}\Delta_x u_\lambda.
\end{equation*}
By \eqref{proof_of_xlambda_distance}, we have 
\begin{equation*}
(1+|y|^2)^{\frac{N+2}{2}}\Delta_yu=(1+|x|^2)^{\frac{N+2}{2}}\Delta_xu_\lambda.
\end{equation*}
From \eqref{deltayu_definition} and \eqref{deltaxulambda_definition}, we can conclude that \eqref{detailed_operater_identity} holds.
\end{proof}

 For each fixed $\tau>0$ and $0<\alpha<1$, we define the parabolic cylinder $Q_{\tau}:=B_1(0)\times[\tau-\varepsilon_0, \tau+\varepsilon_0]$, where $\varepsilon_0$ is a small constant. For any points $(x,t),(\Tilde{x},\Tilde{t})\in Q_\tau$, since $u(x,t)\in C^{2,1}(Q_\tau)$, there exist two constants $\xi_t\in(\min\{t,\Tilde{t}\},\max\{t,\Tilde{t}\})$, $\xi_x\in(0,1)$, such that
\begin{eqnarray*}
\frac{\left|u(x,t)-u\left(\Tilde{x},\Tilde{t}\right)\right|}{\left|x-\Tilde{x}\right|^{\alpha}+\left|t-\Tilde{t}\right|^{\frac{\alpha}{2}}}&=&\frac{\left|u(x,t)-u(x,\Tilde{t})+u(x,\Tilde{t})-u(\Tilde{x},\Tilde{t})\right|}{\left|x-\Tilde{x}\right|^{\alpha}+|t-\Tilde{t}|^{\frac{\alpha}{2}}}\\&\leqslant&\frac{\left|u(x,t)-u(x,\Tilde{t})\right|+|u(x,\Tilde{t})-u(\Tilde{x},\Tilde{t})|}{|x-\Tilde{x}|^{\alpha}+|t-\Tilde{t}|^{\frac{\alpha}{2}}}\\&<&\left|u_t(x,\xi_t)\right|\left|t-\Tilde{t}\right|^{1-\frac{\alpha}{2}}+\left|\nabla u(x+\xi_x(\Tilde{x}-x),t)\right|\left|x-\Tilde{x}\right|^{1-\alpha}.
\end{eqnarray*}
Again by $u(x,t)\in C^{2,1}(Q_\tau)$ and the boundedness of $Q_\tau$, we have $\left|u_t(x,\xi_t)\right|$, $\left|\nabla u(x+\xi_x(\Tilde{x}-x),t)\right|$, $\left|t-\Tilde{t}\right|^{1-\frac{\alpha}{2}}$ and $\left|x-\Tilde{x}\right|^{1-\alpha}$ are all bounded. Then for some $0<\alpha<1$, there exists a constant $C_0>0$ such that 
\begin{equation}\label{up-uq}
u(x,t)-u\left(\Tilde{x},\Tilde{t}\right)\leqslant C_0\left(\left|x-\Tilde{x}\right|^{\alpha}+\left|t-\Tilde{t}\right|^{\frac{\alpha}{2}}\right).
\end{equation}
Combining \eqref{up-uq} with the boundedness of $u(x,t)$ in $Q_\tau$, for any $\tau>0$, the orbit $\{u(\cdot,t),\,t\in[\tau-\varepsilon_0,\tau+\varepsilon_0]\}$ is relatively compact in $C_0(B_1(0))$. 
To directly address the properties of the functions in $\omega(u)$, we present the following lemma.
\begin{lem}\label{Convergence_of_sequences}
Let $M:=sup\{\Arrowvert u(\cdot,t)\Arrowvert_{L^{\infty}}:t>0\}$, under the definition \eqref{uk_definition} and \eqref{fk_definition}, for each $k$, $u_k$ and $f_k$ satisfy the problem \eqref{uk_problem}. 

Then there exist some functions $u_\infty\in C^{2,1}(B_1(0)\times[1-\varepsilon_0,1+\varepsilon_0])$ and $f_\infty\in C((0,1)\times[-M,M]\times[1-\varepsilon_0,1+\varepsilon_0])$ such that $u_k\to u_\infty$ in the sense of $C^{2,1}(B_1(0)\times[1-\varepsilon_0,1+\varepsilon_0])$ and $f_k\to f_\infty$ as $k\to\infty$ in the sense of $C((0,1)\times[-M,M]\times[1-\varepsilon_0,1+\varepsilon_0])$ . And with a constant $\varepsilon_0$, $u_\infty(x,t)$ satisfies 
\begin{equation}\label{uinfty_problem}
\left\{\begin{array}{ll}
\frac{\partial u_{\infty}}{\partial t}-\Delta u_\infty=f_\infty(|x|,u_\infty(x,t),t), &(x,t)\in B_1(0)\times[1-\varepsilon_0,1+\varepsilon_0],  \\
u_\infty=0 & (x,t)\in \partial B_1(0)\times[1-\varepsilon_0,1+\varepsilon_0].
\end{array}\right .
\end{equation}
\end{lem}

\begin{proof}
For each $K$ and any $(x,t),(u,\Tilde{u}),(\Tilde{x},\Tilde{t})$ in the bounded domain $Q_u:=(0,1)\times[-M,M]\times[1-\varepsilon_0, 1+\varepsilon_0]$, by \eqref{fp-fq} and the definition \eqref{fk_definition}, we set $C'$ be the maximum of $C_u$,$C_x$ and $C_t$ appeared in \eqref{fp-fq}, with any $\varepsilon>0$, if $\left||x|-|\Tilde{x}|\right|^{\alpha}+\left|u-\Tilde{u}\right|+\left|t-\Tilde{t}\right|^{\frac{\alpha}{2}}<\delta(\varepsilon)=\varepsilon\backslash C'$, then we have 
\begin{eqnarray*}
|f_k(|x|,u(x,t),t)-f_k(|\Tilde{x}|,u(\Tilde{x},\Tilde{t}),\Tilde{t})|\leqslant C'(|x-\Tilde{x}|^\alpha+|u-\Tilde{u}|+|t-\Tilde{t}|^{\frac{\alpha}{2}})<C'\cdot\frac{\varepsilon}{C'}=\varepsilon,
\end{eqnarray*}
which means the sequence $\left\{f_k\right\}_{k\in\mathbb{N}}$ is equicontinuous in $Q_u$. In addition, owing to the boundedness of $Q_u$ and the continuity of $f$ in $Q_u$, we deduce that $\left\{f_k\right\}_{k\in\mathbb{N}}$ is bounded in $Q_u$ for each $k$. Then the Ascoli-Azela theorem implies that there exists a function $f_\infty\in C(Q_u)$ such that $f_k\to f_\infty$ in the sense of $C(Q_u)$ as $k\to\infty$, and $f_\infty$ is Lipschitz continuous in $u$ by (F1).

Next we discuss the convergence of $\{u_k\}_{k\in\mathbb{N}}$. Set $Q_1:=B_1(0)\times[1-\varepsilon_0,1+\varepsilon_0]$, from \eqref{uk_problem}, for each $k$, let $\Tilde{f}_k(x,t)=f_k(|x|,u_k(x,t),t)$,
then we have
\begin{equation}\label{uk_equation}
\frac{\partial u_k(x,t)}{\partial t}-\Delta u_k(x,t)=\Tilde{f}_k(x,t),\quad(x,t)\in Q_1.
\end{equation}
Additionally, for any $(x,t),(\Tilde{x},\Tilde{t})\in Q_1$, by the condition (F1) and (F2), we have three constants $C_x$, $C_u$ and $C_t$ such that
\begin{equation}\label{fp-fq}
\begin{aligned}
&\quad|\Tilde{f_k}(x,t)-\Tilde{f_k}(\Tilde{x},\Tilde{t})=||f_k(|x|,u_k(x,t),t)-f_k(|\Tilde{x}|,u_k(\Tilde{x},\Tilde{t}),\Tilde{t})|\\&\leqslant|f(|x|,u_k(x,t),t)-f(|\Tilde{x}|,u_k(x,t),t)|+|f(|\Tilde{x}|,u_k(x,t),t)-f(|\Tilde{x}|,u_k(\Tilde{x},\Tilde{t}),t)|\\&\quad+|f(|\Tilde{x}|,u_k(\Tilde{x},\Tilde{t}),t)-f(|\Tilde{x}|,u_k(\Tilde{x},\Tilde{t}),\Tilde{t})|\\
&\leqslant C_x|x-\Tilde{x}|^\alpha+ C_{u}|u_k(x,t)-u_k(\Tilde{x},\Tilde{t})|+C_t|t-\Tilde{t}|^{\frac{\alpha}{2}}.
\end{aligned}
\end{equation}
By \eqref{up-uq}, we set $C=\max\{C_x, C_uC_0,C_t\}$, then for any $(x,t),(\Tilde{x},\Tilde{t})\in Q_1$, we have
\begin{equation}\label{fCalpha}
|\Tilde{f_k}(x,t)-\Tilde{f_k}(\Tilde{x},\Tilde{t})|\leqslant C\left(\left|x-\Tilde{x}\right|^{\alpha}+\left|t-\Tilde{t}\right|^{\frac{\alpha}{2}}\right).
\end{equation}
which means for any two points $P(x,t),\Tilde{P}(\Tilde{x},\Tilde{t})$, we define $d(P,\Tilde{P})=(|x-\Tilde{x}|+|t-\Tilde{t}|^{\frac{1}{2}})$, then we can get that 
\begin{equation}
\mathop{sup}_{\substack{P,\Tilde{P}\in Q_1\\P\not=\Tilde{P}}}\frac{|\Tilde{f}_k(P)-\Tilde{f}_k(\Tilde{P})|}{d^{\alpha}(P,\Tilde{P})}<+\infty.
\end{equation}
Given that $u_k$ is the solution of \eqref{uk_problem}, then the existence and uniqueness theorem of classical solutions for heat equations(Theorem 3.3.7 in \cite{friedman-2008}) implies that 
\begin{equation}\label{D2uk_equicontinuous}
\begin{aligned}
&\,\quad|u_k|_{2+\alpha,1+\frac{\alpha}{2};Q_1}\\&:=\mathop{sup}_{\substack{P\in Q_1}}|u_k(P)|+\mathop{sup}_{\substack{P\in Q_1}}|Du_k(P)|+\mathop{sup}_{\substack{P\in Q_1}}|D^2u_k(P)|+\mathop{sup}_{\substack{P\in Q_1}}\left|\frac{\partial u_k}{\partial t}(P)\right|+\mathop{sup}_{\substack{P,\Tilde{P}\in Q_1\\P\not=\Tilde{P}}}\frac{|u_k(P)-u_k(\Tilde{P})|}{d^{\alpha}(P,\Tilde{P})}\\&\quad+\mathop{sup}_{\substack{P,\Tilde{P}\in Q_1\\P\not=\Tilde{P}}}\frac{|Du_k(P)-Du_k(\Tilde{P})|}{d^{\alpha}(P,\Tilde{P})}+\mathop{sup}_{\substack{P,\Tilde{P}\in Q_1\\P\not=\Tilde{P}}}\frac{|D^2u_k(P)-D^2u_k(\Tilde{P})|}{d^{\alpha}(P,\Tilde{P})}+\mathop{sup}_{\substack{P,\Tilde{P}\in Q_1\\P\not=\Tilde{P}}}\frac{\left|\frac{\partial u_k}{\partial t}(P)-\frac{\partial u_k}{\partial t}(\Tilde{P})\right|}{d^{\alpha}(P,\Tilde{P})}\\&\leqslant C<+\infty,
\end{aligned}
\end{equation}
where $C$ is a constant independent of $k$, $Du_k=\left(\partial u_k/\partial x_1,\partial u_k/\partial x_2,\cdots,\partial u_k/\partial x_N\right)$ and $D^2u_k$ is the Hessian matrix of $u_k$. On the one hand \eqref{D2uk_equicontinuous} shows the equicontinuity of $u_k$, $Du_k$, $D^2u_k$ and $\partial u_k/\partial t$ in $C(Q_1)$, which means for a given $\varepsilon>0$, we can choose $\delta>0$ independent of $k$ such that 
\begin{equation*}
    |u_k(P)-u_k(\Tilde{P})|+\left|\frac{\partial u_k}{\partial t}(P)-\frac{\partial u_k}{\partial t}(\Tilde{P})\right|+|Du_k(P)-Du_k(\Tilde{P})|+|D^2u_k(P)-D^2u_k(\Tilde{P})|\leqslant C\delta^\alpha<\varepsilon,
\end{equation*} for all $P,\Tilde{P}\in Q_1$ and $k$. On the other hand \eqref{D2uk_equicontinuous} also ensures that for each $k$, $u_k$, $Du_k$, $D^2u_k$ and $\partial u_k/\partial t$ are all bounded in $Q_1$. By the Ascoli-Azela theorem, there exist functions $u_\infty$, $u_{t\infty}$,  $\left\{u_{i\infty}\right\}_{i=1}^N$, $\left\{u_{ij\infty}\right\}_{i,j=1}^N\in C(Q_1)$ such that as $k\to\infty$ for each $i$ and $j$ we have
\begin{equation*}
    u_k\to u_\infty,\quad\frac{\partial u_k}{\partial t}\to u_{t\infty},\quad \frac{\partial u_k}{\partial x_i}\to u_{i\infty},\quad \frac{\partial u_k}{\partial x_i\partial x_j}\to u_{ij\infty}
\end{equation*}
which are all in the sense of $C(Q_1)$. Employing the fundamental theorem of caculus and \eqref{D2uk_equicontinuous}, for $\varepsilon>0$, $(x,t)\in Q_1$ and fixed $i$, we can choose $\overline{\delta}=\min(\delta,1)>0$ independent of $k$ such that $B_{\overline{\delta}}(x)\times[1-\varepsilon_0,1+\varepsilon_0]\in Q_1$ and for all $|h|<\overline{\delta}$ and $k$ we have
\begin{equation*}
\begin{aligned}
&\quad \left|\frac{\partial u_k}{\partial x_i}(x_1,\cdots,x_j+h,\cdots,x_N,t)-\frac{\partial u_k}{\partial x_i}(x_1,\cdots,x_j,\cdots,x_N,t)-\frac{\partial^2u_k}{\partial x_i\partial x_j}(x,t)h\right|\\&=\left|\int^1_0 \frac{\partial^2u_k}{\partial x_i\partial x_j}(x_1,\cdots,x_j+sh,\cdots,x_N,t)h ds-\frac{\partial^2u_k}{\partial x_i\partial x_j}(x,t)h\right|\\&\leqslant\int_0^1\left|\frac{\partial^2u_k}{\partial x_i\partial x_j}(x_1,\cdots,x_j+sh,\cdots,x_N,t)-\frac{\partial^2u_k}{\partial x_i\partial x_j}(x,t)\right||h|ds\leqslant|h|^{1+\alpha}C<C\delta^\alpha|h|<\varepsilon|h|.
\end{aligned}    
\end{equation*}
Similarly we can get that
\begin{align*}
&\left|u_k(x_1,\cdots,x_i+h,\cdots,x_N,t)-u_k(x,t)-
\frac{\partial u_k}{\partial x_i}(x,t) h\right|\leqslant\mathop{sup}_{Q_1}\left|D^2 u_k\right||h|^{1+\alpha}\leqslant C|h|^{1+\alpha}<\varepsilon |h|,\\ &\quad\left|u_k(x,t+h)-u_k(x,t)-\frac{\partial u_k(x,t)}{\partial t} h\right|\leqslant\int_0^1\left|\frac{\partial u_k}{\partial t}(x,t+sh)-\frac{\partial u_k}{\partial t}(x,t)\right||h|ds<\varepsilon|h|.
\end{align*}
Letting $k\to\infty$, we can derive that
\begin{equation}\label{partialu_limit}
    \begin{aligned}
        &|u_{i\infty}(x_1,\cdots,x_j+h_j,\cdots,x_N,t)-u_{i\infty}(x_1,\cdots,x_j,\cdots,x_N,t)-u_{ij\infty}(x,t)h|<\varepsilon |h|,\\ &|u_\infty(x_1,\cdots,x_i+h_i,\cdots,x_N,t)-u_\infty(x_1,\cdots,x_i,\cdots,x_N,t)-u_{i\infty}(x,t)h|<\varepsilon |h|\\&\qquad\qquad\left|u_\infty(x,t+h)-u_\infty(x,t)-u_{t\infty}h\right|<\varepsilon |h|.
    \end{aligned}
\end{equation}
Then \eqref{partialu_limit} means that for each $i$ and $j$
\begin{equation}\label{uinC21}
u_{ij\infty}=\frac{\partial^2u_\infty}{\partial x_i\partial x_j},\quad u_{i\infty}=\frac{\partial u_\infty}{\partial x_i},\quad u_{t\infty}=\frac{\partial u_\infty}{\partial t},
\end{equation}
as $h\to 0$. The above discussion shows  that $u_\infty\in C^{2,1}(Q_1)$ and as $k\to\infty$ we observe that $u_k \to u_\infty$ in the sense of $C^{2,1}(Q_1)$ and 
\begin{equation}\label{ukpartials_limitderivation}
\frac{\partial u_k}{\partial t}\to \frac{\partial u_\infty}{\partial t},\quad Du_k\to Du_\infty,\quad D^2u_k\to D^2u_\infty,
\end{equation}
all in the sense of $C(Q_1)$. By \eqref{ukpartials_limitderivation} and the existence of $u_\infty$ and $f_\infty$, consider the problem \eqref{uk_problem} as $k\to\infty$, we can deduce that $u_\infty$ satisfies
\begin{equation*}
\left\{
\begin{aligned}
&\frac{\partial u_\infty}{\partial t}(x,t)-\Delta u_\infty(x,t)=f_\infty(|x|,u_\infty(x,t),t), \quad&(x,t)\in B_1(0)\times[1-\varepsilon_0,1+\varepsilon_0], \\
& u_\infty(x,t)=0, \quad & (x,t)\in \partial B_1(0)\times[1-\varepsilon_0,1+\varepsilon_0],
\end{aligned}\right.
\end{equation*}
Then we complete the proof of Lemma \ref{Convergence_of_sequences}.
\end{proof}

\textbf{Disclosure statement}

The author reports there are no competing interests to declare.


\begin{thebibliography}{9}

\bibitem{zhang_yi_wu-2024}
Zhang, Y., Yi, T., Wu, J. (2024). Global population propagation dynamics of reaction-diffusion models with shifting environment for non-monotone kinetics and birth pulse. \emph{J. Differ. Equ.} 402:290-314.

\bibitem{jia_wang-2023}
Jia, F., Wang, Z. (2023). The stability of diverging traveling fronts and threshold phenomenon for the buffered bistable system. \emph{J. Differ. Equ.} 356:59-110.

\bibitem{chen_polacik-1996}
Chen, X., Poláčik, P. (1996). Asymptotic periodicity of positive solutions of reaction diffusion equations on a ball. \emph{J. Reine Angew. Math.} 472:17-51.

\bibitem{alexandrov-1962}
Alexandrov, A. D. (1962). A characteristic property of spheres. \emph{Ann. Mat. Pura. Appl.} 58(4):303-315.	

\bibitem{serrin-1971}
Serrin, J. (1971). A symmetry problem in potential theory. \emph{Arch. Rational Mech. Anal.} 43(4):304-318.

\bibitem{Bre1991}
Berestycki, H., Nirenberg, L. (1991). On the method of moving planes and the sliding method. \emph{Bol. Soc. Brasil. Mat.} 22(1):1–37.

\bibitem{gidas_ni_nirenberg-1979}
Gidas, B., Ni, W., Nirenberg, L. (1979). Symmetry and related properties by the maximum principle. \emph{Comm. Math. Phys.} 68(3):209-243.

\bibitem{ChenLi1997}
Chen, W., Li, C. (1997). A priori estimates for prescribing scalar curvature equations. \emph{Ann. of Math.} 145:547–564.

\bibitem{Chen2016}
Chen, L., Liu, Z., Lu, G. (2016). Symmetry and regularity of solutions to the weighted Hardy-Sobolev type system. \emph{Adv. Nonlinear Stud.} 16(1):1–13.

\bibitem{Dou2019}
Dou, J., Zhu, M. (2019). Nonlinear integral equations on bounded domains. \emph{J. Funct. Anal.} 277(1):111–134.

\bibitem{Guo2019}
Guo, Q. (2019). Blowup analysis for integral equations on bounded domains. \emph{J. Differ. Equ.} 266(12):8258–8280.

\bibitem{Guo2008}
Guo, Y., Liu, J. (2008). Liouville type theorems for positive solutions of elliptic system in $\mathbb{R}^N$. \emph{Comm. Partial Differential Equations} 3(2):263-284.

\bibitem{LiNi1993}
Li, Y., Ni, W. (1993). Radial symmetry of positive solutions of nonlinear elliptic equations in $\mathbb{R}^n$. \emph{Comm. Partial Differential Equations} 18(5-6): 1043-1054.
\bibitem{Bre1988}
Berestycki, H., Nirenberg, L. (1988). Monotonicity, symmetry and antisymmetry of solutions of semilinear elliptic equations. \emph{J. Geom. Phys.} 5(2):237–275.

\bibitem{WeiXu1999}
Wei, J., Xu, X. (1999). Classification of solutions of higher order conformally invariant equations. \emph{Math. Ann.}  313 (2):207-228

\bibitem{CaoDai}
Cao, D., Dai, W. (2018). Classification of nonnegative solutions to a bi-harmonic equation with Hartree type nonlinearity. \emph{Proc. Roy. Soc. Edinburgh-A: Math.} 149(4): 979-994.

\bibitem{Chen2015}
Chen, W., Fang, Y., Yang, R. (2015). Liouville theorems involving the fractional Laplacian on a half space. \emph{Adv. Math.}  274:167-198


\bibitem{busca_sirakov-2000}
Busca, J., Sirakov, B. (2000). Symmetry results for semilinear elliptic systems in the whole space. \emph{J. Differ. Equ.} 163(1):41-56.

\bibitem{Chen2017}
Chen, W., Li, C., Li, Y. (2017). A direct method of moving planes for the fractional Laplacian. \emph{Adv. Math.} 308:404–437.

\bibitem{LiuXu}
Liu B., Xu S. (2024). Liouville-Type Theorem for an Indefinite Logarithmic Laplacian Equation. \emph{International Journal of Mathematics}.

\bibitem{Liu2018}
Liu, B. (2018). Direct method of moving planes for logarithmic Laplacian system in bounded domains. \emph{Disc. Cont. Dyn. Sys.} 38(10):5339-5349.

\bibitem{ChLZh2017}
Chen, W., Li, Y., Zhang, R. (2017). A direct method of moving spheres on fractional order equations. \emph{J. Funct. Anal.} 272(10):4131–4157.

\bibitem{Dou2017}
Dou, J., Guo, Q., Zhu, M. (2017). Subcritical approach to sharp hardy-littlewood-Sobolev type inequalities on the upper half space. \emph{Adv. Math.} 312:1–45.

\bibitem{Jin2011}
Jin, T. L. (2011). Symmetry and nonexistence of positive solutions of elliptic equations and systems with Hardy terms. \emph{Ann. Inst. H. Poincaré C Anal. Non Linéaire} 28(6):965–981.

\bibitem{LiZh1995}
Li, Y., Zhu, M. (1995). Uniqueness theorems through the method of moving spheres. \emph{Duke Math. J.} 80(2):383–417.

\bibitem{LiuZ2021}
Liu, Z. (2021). Maximum principles and monotonicity of solutions for fractional p-equations in unbounded domains. \emph{J. Differ. Equ.} 270:1043–1078. 

\bibitem{Wu2020}
Wu, L., Chen, W. (2020). The sliding methods for the fractional p-Laplacian. \emph{Adv. Math.} 361:106933.

\bibitem{Chenbook}
Chen, W., Li, C. (2010). \emph{Methods on Nonlinear Elliptic Equations}, volume 4 of Diff. Equa. Dyn. Sys. in AIMS Book Series.

\bibitem{Chenbook2}
Chen, W., Li, Y., Ma, P. (2020). \emph{The fractional laplacian}. World Scientific.

\bibitem{naito_nishimoto_suzuki-1996}
Naito, Y., Nishimoto, T., Suzuki, T. (1996). Radial symmetry of positive solutions for semilinear elliptic equations in a disc. \emph{Hiroshima Math. J.} 26(3):531-545.

\bibitem{naito_suzuki-1998}
Naito, Y., Suzuki, T. (1998). Radial symmetry of positive solutions for semilinear elliptic equations on the unit ball in $\mathbb{R}^N$. \emph{Funkcial. Ekvac.} 41(2):215-234.

\bibitem{shioji_watanabe-2012}
Shioji, N., Watanabe, K. (2012). Radial symmetry of positive solutions for semilinear elliptic equations in the unit ball via elliptic and hyperbolic geometry. \emph{J. Differ. Equ.} 252(2):1392-1402.

\bibitem{babin_sell-2000}
Babin, A. V., Sell, G. R. (2000). Attractors of non-autonomous parabolic equations and their symmetry properties. \emph{J. Differ. Equ.} 160(1):1-50.

\bibitem{dancer_hess-1994}
Dancer, E. N., Hess, P. (1994). The symmetry of positive solutions of periodic parabolic problems. \emph{J. Comput. Appl. Math.} 52(1-3):82-89.

\bibitem{foldes-2011}
Földes, J. (2011). On symmetry properties of parabolic equations in bounded domains. \emph{J. Differ. Equ.} 250(12):4236-4261.

\bibitem{babin-1994}
Babin, A. V. (1994). Symmetrization properties of parabolic equations in symmetric domains. \emph{J. Dynam. Differ. Equ.} 6(4):639-658.

\bibitem{polacik-2006}
Poláčik, P. (2006). Symmetry properties of positive solutions of parabolic equations on $\mathbb{R}^N$: II. Entire solutions. \emph{Comm. Partial Differ. Equ.} 31(11):1615-1638.

\bibitem{li-1989}
Li, C. (1989). Some qualitative properties of fully nonlinear elliptic and parabolic equations. New York University.

\bibitem{hess_polacik-1994}
Hess, P., Poláčik, P. (1994). Symmetry and convergence properties for nonnegative solutions of nonautonomous reaction-diffusion problems. \emph{Proc. Roy. Soc. Edinburgh} 124(3):573-587.

\bibitem{polacik-2007}
Poláčik, P. (2007). Estimates of solutions and asymptotic symmetry for parabolic equations on bounded domains. \emph{Arch. Rational Mech. Anal.} 183(1):59-91.

\bibitem{polacik-2005}
Poláčik, P. (2005). Symmetry properties of positive solutions of parabolic equations on $\mathbb{R}^N$: I. Asymptotic symmetry for the Cauchy problem. \emph{Comm. Partial Differ. Equ.} 30(11):1567-1593.

\bibitem{babin_sell-1995}
Babin, A. V., Sell, G. R. (1995). Symmetry of instabilities for scalar equations in symmetric domains. \emph{J. Differ. Equ.} 123(1):122-152.

\bibitem{polacik-2009}
Poláčik, P. (2009). Symmetry properties of positive solutions of parabolic equations: a survey, in: Recent progress on reaction-diffusion systems and viscosity solutions. \emph{World Scientific 2009} 170-208.

\bibitem{saldana_weth-2012}
Saldaña, A., Weth, T. (2012). Asymptotic axial symmetry of solutions of parabolic equations in bounded radial domains. \emph{J. Evol. Equ.} 12(3):697-712.

\bibitem{chen_wang_niu-2021}
Chen, W., Wang, P., Niu, Y., Hu, Y. (2021). Asymptotic method of moving planes for fractional parabolic equations. \emph{Adv. Math.} 337:107463.

\bibitem{evans}
Evans, L. C. (2022). \emph{Partial Differential Equations}, 2nd ed., volume 19 of Graduate Studies in Mathematics. American Mathematical Society.

\bibitem{friedman-1958}
Friedman, A. (1958). Remarks on the maximum principle for parabolic equations and its applications. \emph{Pacific J. Math.} 8(2):201-211.

\bibitem{friedman-2008}
Friedman, A. (2008). \emph{Partial differential equations of parabolic type}. Courier Dover Publications.

\end{thebibliography}
\end{document}